\documentclass[11pt]{article}
\usepackage{amsmath,amsthm,amsfonts,amssymb,amscd, amsxtra,color}
\usepackage[normalem]{ulem}

\usepackage[margin=1.0in]{geometry}

\theoremstyle{plain}
\newtheorem{theorem}{Theorem}
\newtheorem{lemma}{Lemma}
\newtheorem{definition}{Definition}
\newtheorem{corollary}{Corollary}
\newtheorem{proposition}{Proposition}

\newtheorem{remark}{Remark}

\newenvironment{proof4}{{\noindent\bf Proof of Theorem 4.}}{\hfill$\Box$\\}
\begin{document}
%%%%%%%%%%%%%%%%%%%%%%%%%%%%%%%%%%%%%%%%%%%%%%%%%%%%%%%%%%%%%
\title{\textbf{Robust  Kantorovich's theorem on Newton's method under  majorant condition in Riemannian Manifolds}}

\author{ T. Bittencourt \thanks{IME/UFG,  CP-131, CEP 74001-970 - Goi\^ania, GO, Brazil (Email: {\tt
      tiberio.b@gmail.com}). This author was supported  by CAPES.}
\and
O. P. Ferreira\thanks{IME/UFG,  CP-131, CEP 74001-970 - Goi\^ania, GO, Brazil (Email: {\tt
      orizon@mat.ufg.br}). This author was partially supported by FAPEG 201210267000909 - 05/2012, PRONEX-Optimization(FAPERJ/CNPq), CNPq Grants 4471815/2012-8,  305158/2014-7.}
}
\date{May  19, 2015}
\maketitle
%%%%%%%%%%%%%%%%%%%%%%%%%%%%%%%%%%%%%%%%%%%%%%%%%%%
\begin{abstract}
%%%%%%%%%%%%%%%%%%%%%%%%%
A robust affine invariant version  of Kantorovich's theorem on Newton's method,  for   finding a zero of a differentiable vector field defined on a complete Riemannian manifold,  is presented in this paper.  In the analysis presented, the  classical Lipschitz condition is relaxed  by  using a general  majorant function, which allow  to  establish existence and local uniqueness of the solution  as well as unifying  previously results pertaining Newton's method.   The most important in our analysis is the robustness,  namely,   is given  a  prescribed ball,  around the point  satisfying Kantorovich's assumptions,  ensuring convergence of the method for any starting point in this ball. Moreover,   bounds for $Q$-quadratic convergence of  the method which depend  on the  majorant function is obtained.

\end{abstract}

\maketitle
%%%%%%%%%%%%%%%%%%%%%%%%%%%%%%%%%%%%%%%%%%%%%%%%%%%%%%%%
\section{Introduction}\label{sec:int}
%%%%%%%%%%%%%%%%%%%%%%%%%%%%%%%%%%%%%%%%%%%%%%%%%%%%%%%%
Extension  of   concepts and techniques  as well as methods of Mathematical Programming   from  the  Euclidean space  to  Riemannian setting it is natural  and  has been done frequently before; see, e.g.,\cite{Absil2014,  AdlerDedieuShub2002,  AABCP2014-3, FerreiraSvaiter2002, Manton2015,  Smith1994, WenWotao2013}.    The motivation of this extensions,  which in general is nontrivial,   is either of purely theoretical nature or aims at obtaining  efficient algorithms; see, e.g., \cite{ Absil2014,  AdlerDedieuShub2002, EdelmanAriasSmith1999, Karmarkar1998, NesterovTodd2002, Manton2015, MillerMalick2005,   Smith1994, WenWotao2013}.  Indeed, many optimization problems are naturally posed on Riemannian manifolds, which has a specific underlying geometric and algebraic structure that could be exploited to greatly reduce the cost of obtaining the solutions. For instance,  in order to take advantage of the Riemannian geometric  structure,   it is suitable to treat  some constrained optimization problems as one of finding the zeros of a gradient vector field  on a Riemannian manifolds rather than use the method of Lagrange multipliers or projection idea for solving the problem; see \cite{Absil2014,  AdlerDedieuShub2002, Manton2015,  Smith1994, WenWotao2013}.  In this case, constrained optimization problems can be seen as unconstrained one from the Riemannian geometry viewpoint. Besides, the Riemannian geometry allows  to induce new research directions  so as to   produce competitive algorithms;  see \cite{EdelmanAriasSmith1999, Karmarkar1998, NesterovTodd2002, Smith1994}. In this paper, instead of considering the problem of finding the zero of the gradient field  on a Riemannian manifolds, let us consider the more general problem of  finding a zeros of a  vector field defined on a Riemannian manifold.

On the other hand, the Newton's  method and its variant are  powerful tools for finding a zero of  nonlinear function  in real or complex Banach space. Besides its practical applications, Newton's method is  also a powerful theoretical tool having a wide range of applications in pure mathematics; see \cite{Blum1998, Hamilton1982,  KrantzParks2013, Moser1961, Nash1956, Wayne1996}.  Therefore,  a couple of papers have dealt with the issue of  generalization of Newton's  method   and its variant   from Euclidean  to  Riemannian  setting in order to go further in the study of the convergence properties of this method.  Early works dealing with the generalization of Newton's methods to  Riemannian setting include \cite{ DedieuShub2000, EdelmanAriasSmith1999, Gabay1982,  Owren2000,  Shub1986, Udriste1994}.   Actually, the generalization  of   Newton's  method  to  Riemannian setting has  been done with several different purposes, including the purpose of finding a zeros of a gradient vector field or,  more generally,  with  the purpose of  finding a zero of a differentiable vector field; see  \cite{Absil2014, AdlerDedieuShub2002, Argyros2007,  ArgyroHilout2009,   ArgyrosMagrenan2014,     DedieuPriouretMalajovich2003,  DedieuShub2000,  EdelmanAriasSmith1999,   FerreiraRoberto2012, FerreiraSvaiter2002,  LiWang2006, LiJinhua2008, LiWangDedieu2009, Manton2015, MillerMalick2005, RingBenedikt2012,  Schulz2014, Smith1994, Wang2011,  WangHuangChong2009,  WangJenLi2012, Zhang2010} and the references therein.

 Properties of convergence of  Newton's method have been extensively studied on several papers due to the important role that it plays in the development of  numerical methods for finding a zero of a differentiable vector field defined on a complete Riemannian manifolds.  In   2002   Ferreira and Svaiter in \cite{FerreiraSvaiter2002} extended  the  Kantorovich's theorem on the Newton's method  to Riemannian setting  using a new technique which simplifies the analysis and proof of this theorem.    It is worth mention   that, in a similar spirit,    an  extensions of the famous Smale's  theory; see \cite{Smale1986},  to analytic vector fields  on analytic Riemannian manifolds were done in 2003 by Dedieu et al. in \cite{DedieuPriouretMalajovich2003}. The basic idea of \cite{FerreiraSvaiter2002} was  to combine a formulation of Kantorovich's theorem by means of quadratic  majorant functions, see \cite{ZabrejkoNguen1987} for more general majorant functions,  with the definitions of good regions for the  Newton's method. In these regions, the majorant function bounds the vector fields  which the zero  is to be found, and the behavior of the  Newton's iteration in these regions is estimated using iterations associated to the majorant function. Moreover, as a whole, the union of all these regions is invariant under  Newton's iteration.  Afterward, this technique was successfully employed for proving generalized versions of Kantorovich's theorem in Riemannian setting.  Inspired by previous work of Zabrejko and Nguen in \cite{ZabrejkoNguen1987} on Kantorovich's majorant method,  a radial parametrization of a Lipschitz-type   and   L-average Lipschitz  affine invariant majorante conditions were introduced in Riemannian setting  by   Alvarez et al. in \cite{Alvarez2008} and  Li and Wang  in  \cite{LiJinhua2008}, respectively, in order to  establish existence and local uniqueness of the solution as well as unifying  previously  convergence criterion of Newton's method.

In the present paper, we will use the technique introduced in \cite{FerreiraSvaiter2002}, see also \cite{FerreiraSvaiter2012},  to present a robust affine invariant version  of the Kantorovich's theorem on the Newton's method  finding a zeros of a differentiable vector field defined on a complete Riemannian manifold.  In our analysis, the classical Lipschitz condition is relaxed  using a general  majorant function.  The analysis presented provides a clear relationship between the majorant function  and the vector field under consideration.  However,   the most important in our analysis is the robustness,  namely,   we give a  prescribed ball,  around the point  satisfying the Kantorovich's assumptions,  ensuring convergence of the method for any starting point in this ball. Moreover,  we  establish bounds for $Q$-quadratic convergence of  the method which depend  on the  majorant function.    Also, as in \cite{Alvarez2008} and  \cite{LiJinhua2008},  this analysis  allows us   establish existence and local uniqueness of the solution as well as unifying  previously results pertaining Newton's method.

The organization of the paper is as follows. In Section~\ref{sec:aux},  some notations and one basic results used in the paper are presented. In Section~\ref{SeC:RobKantTheo}, the main result is stated, namely, the robust  affine invariant  Kantorovich's theorem for Newton's method and  in   Section~\ref{SeC:RobKantTheo}  the  affine invariant version, which is used for proving the robust one is stated and proved. In Section~\ref{pr:mr} we prove the main theorem.  In  Section \ref{SEC:SpecCase} three special case of the main theorem is presented.  Some final remarks are made in Section \ref{Sec:FinalRemarl}.

%%%%%%%%%%%%%%%%%%%%%%%%%%%%%%%%%%%%%%%%%%%%%%%%%%%%%%%%%%%%%%%%%%%%%%%%%%%%%%%%%%%%%%%%%%%%%
\subsection{Notation and auxiliary results} \label{sec:aux}
%%%%%%%%%%%%%%%%%%%%%%%%%%%%%%%%%%%%%%%%%%%%%%%%%%%%%%%%%%%%%%%%%%%%%%%%%%%%%%%%%%%%%%%%%%%%

In this section we recall some notations, definitions and basic properties of Riemannian manifolds used throughout the paper,  they can be found, for example, in \cite{DoCa92}  and \cite{La95}.

Throughout the paper, $\cal{M}$ is a smooth manifold and  $C^{1}(\cal M)$ is the class of all  continuously differentiable functions  on  $\cal M$. The space of vector fields  $C^{r}(\cal M)$  on $\cal M $ is denoted by ${\mathcal X}^{r}(\cal M)$,   by $T_{p}{\cal M}$ we denote the tangent space of $\cal M$ at $p$ and by $T{\cal M}={\bigcup_{x\in {\cal M}}}\,T_{x}{\cal M}$ the {\it tangent bundle \/} of $\cal M$.
Let $\cal M$ be endowed with a Riemannian metric $ {\langle} \cdot , \cdot {\rangle}$, with corresponding norm denoted by $\|\cdot\|$, so that $\cal M$ is now a {\it Riemannian manifold}. Let us recall that the metric can be used to define the length of a piecewise $C^{1}$ curve $\zeta :\, [a,b] \rightarrow {\cal M}$ by
$$
\ell[\zeta, a,b] := \int_a^b \|\zeta^{\prime}(t)\| dt.\
$$
Minimizing this length functional over the set of all such curves we obtain a distance $d(p,q)$, which induces the original topology on $\cal M$. The open and closed  balls of radius $r>0$ centered at $p$ are  defined, respectively, as
$$
B(p,r):=\left\{q\in M:d(p,q) <r\right\}, \qquad B[p,r]:=\left\{q\in M:d(p,q) \leq r\right\}.
$$
 Let $\zeta$ be a curve joining the points $p$ and $q$ in $\cal M$ and let
$\nabla$ be the Levi-Civita connection associated to $(\cal M,{\langle}, {\rangle})$. For each $t \in [a,b]$, $\nabla$ induces an isometry, relative to $ \langle \cdot , \cdot \rangle  $,
\begin{equation} \label{D:dtp}
\begin{split}
P_{\zeta,a,t} \colon T _{\zeta(a)} {\cal M} &\longrightarrow T _
{\zeta(t)} {\cal M}\\
v &\longmapsto P_{\zeta,a,t}\, v = V(t),
\end{split}
\end{equation}
where $V$ is the unique vector field on $\zeta$ such that
$ \nabla_{\zeta'(t)}V(t) = 0$ and $V(a)=v$,
the so-called {\it parallel transport} along $\zeta$ from $\zeta(a)$ to $\zeta(t)$. Note also that
\begin{equation} \label{eq:ptp}
P_{\zeta,b_1,b_2}\circ P_{\zeta,a,b_1} = P_{\zeta,a,b_2}, \qquad P_{\zeta,b,a}  = {P_{\zeta,a,b} } ^{-1}.
\end{equation}
A vector field $V$ along $\zeta$ is said to be {\it parallel} if $\nabla_{\zeta^{\prime}} V=0$. If $\zeta^{\prime}$ itself is parallel, then we say that $\zeta$ is a {\it geodesic}. The geodesic equation $\nabla_{\ \zeta^{\prime}} \zeta^{\prime}=0$ is a second order nonlinear ordinary differential equation, so the geodesic $\zeta$ is determined by its position $p$ and velocity $v$ at $p$. The restriction of a geodesic to a closed bounded interval is called a {\it geodesic segment}. It is easy to check that $\|\zeta ^{\prime}\|$ is constant. We usually do not distinguish between a geodesic and its geodesic segment, as no confusion can arise. We say that $ \zeta $ is {\it normalized} if $\| \zeta^{\prime}\|=1$.  A geodesic $\zeta: [a, b]\to {\cal M}$ is said to be {\it minimal} if its length is equal the distance of its end points, i.e. $\ell[\zeta, a,b]=d(\zeta(a), \zeta(b))$.

A Riemannian manifold is {\it complete} if its geodesics are defined for any values of $t$. {\it
In  this paper, all manifolds  $\mathcal{M}$ are assumed to be complete}.   The Hopf-Rinow's
theorem asserts that if this is the case then any pair of points, say $p$ and $q$, in $\mathcal{M}$ can be joined by
a (not necessarily unique) minimal geodesic segment. Moreover, $({\cal M}, d)$ is a complete metric space
and bounded and closed subsets are compact. The {\it exponential map} at $p,$ $\mbox{exp}_{p}:T_{p}  {\cal M} \rightarrow {\cal M} $, is defined by $\mbox{exp}_{p}v\,=\, \zeta _{v}(1)$, where $\zeta _{v}$ is the geodesic defined by its position $p$ and velocity $v$ at $p$ and  $ \,\zeta _{v}(t)\,=\,\mbox{exp}_{p}tv$ for any value of $t$.

Let $X \in C^{1}(\cal M)$. The  covariant derivative of $X$ determined by
the Levi-Civita connection $\nabla$ defines at each $p\in {\cal M}$ a linear map $\nabla X (p):T_p {\cal M} \to T_p {\cal M}$ given by
\begin{equation}\label{D:ddc}
 \nabla X(p) v:=\nabla_Y X (p),
\end{equation}
where $Y$ is a vector field such that $Y(p)=v$.

\begin{definition} \label{d:scd} Let $Y_1, \ldots, Y_n$ be vector fields on $\cal M$. Then,  the $n$-th covariant derivative of $X$ with respect to $Y_1, \ldots, Y_n$ is defined inductively by
\[
\nabla^{2}_{{\{Y_1, Y_2\}}}X:=\nabla_{{Y_2}}\nabla_{Y_1}X, \qquad
 \nabla^{n}_{{\{Y_i\}_{i=1}^{n}}}X:=\nabla_{{Y_n}}( \nabla_{Y_{n-1}} \cdots \nabla_{Y_{1}}X).
\]
\end{definition}
\begin{definition} \label{d:scdp} Let  $p\in {\cal M}$. Then,  the $n$-th covariant derivative of $X$  at $p$ is the $n$-th multilinear map $\nabla^n X (p):T_p {\cal M}\times \ldots \times T_p {\cal M} \to T_p {\cal M}$ defined by
\[
\nabla^{n}X(p)(v_1, \dots, v_n):=\nabla^{n}_{{\{Y_i\}_{i=1}^{n}}}X(p),
\]
where $Y_1, \ldots, Y_n$ are vector fields on $\cal M$ such that $Y_1(p)=v_1, \ldots, Y_n(p)=v_n$.
\end{definition}
We remark that Definition~\ref{d:scdp} only depends on the $n$-tuple of vectors $(v_1, \ldots, v_n)$ since the covariant derivative is tensorial in each vector field $Y_i$.
\begin{definition} \label{d:norm}
 Let  $p\in {\cal M}$. The norm of an
$n$-th multilinear map $A:T_p {\cal M}\times \ldots \times T_p {\cal M} \to T_p {\cal M}$  is defined by
  \[
 \|A\|=\sup \left\{ \|A(v_1, \dots, v_n) \| \;:\;\; v_1, \dots, v_n\in T_p {\cal M}, \,\|v_i\|=1, \, i=1, \ldots, n \right\}.
 \]
 In particular, the norm of the $n$-th covariant derivative of $X$  at $p$ is given  by
 \[
 \|\nabla^{n}X(p)\|=\sup \left\{\|\nabla^{n}X(p)(v_1, \dots, v_n) \| \;:\;\; v_1, \dots, v_n\in T_p {\cal M}, \,\|v_i\|=1, \, i=1, \ldots, n \right\}.
 \]
\end{definition}
Now, the {\it Fundamental Theorem of Calculus} for a vector field $X$ becomes
\begin{lemma}\label{le:TFC}
Let $\Omega$ be an open subset of $\cal M$, $X$ a
$C^1$
vector field defined on $\Omega$ and  $\zeta:[a, b]\to {\Omega}$  a $C^1$ curve. Then
\[
P_{\zeta,t,a} X(\zeta(t))= X(\zeta(a)) + \int_a^t
   P_{\zeta,s,a} \nabla X (\zeta(s)) \,\zeta'(s) \,\rm{d}s, \qquad t\in [a, b].
\]
\end{lemma}
\begin{proof}
See \cite{FerreiraSvaiter2002}.
\end{proof}
\begin{lemma}\label{le:TFC2}
Let $\Omega$ be an open subset of $\cal M$, $X$ a
$C^2$
vector field defined on $\Omega$ and  $\zeta:[a, b]\to {\Omega}$  a $C^1$ curve. Then for all $Y\in {\mathcal X}(\cal M)$ we have
\[
P_{\zeta,t,a} \nabla X(\zeta(t))\,Y(\zeta(t))= \nabla X(\zeta(a))Y(\zeta(a)) + \int_a^t
   P_{\zeta,s,a} \nabla^{2} X (\zeta(s)) \left(Y(\zeta(s)), \zeta'(s)\right) \,\rm{d}s,  \qquad t\in [a, b].
\]
\end{lemma}
\begin{proof}
See \cite{LiWang2006}.
\end{proof}
\begin{lemma}[Banach's Lemma] \label{lem:ban}
Let $B$ be a linear operator  and  let $I_p$ be the identity operator in  $T _p M$. If  $\|B-I_p\|<1$  then $B$ is nonsingular and
$
\|B^{-1}\|\leq 1/\left(1-  \|B-I_p\|\right).
$
\end{lemma}
\begin{proof}
See, for example,  \cite{Smale1986}.
\end{proof}
We also need the following  elementary convex analysis result, see \cite{HiriartLemarecal1}:
\begin{proposition} \label{pr:conv.aux1}
Let $I\subset \mathbb{R}$ be an interval and $\varphi:I\to \mathbb{R}$ be convex. For any $s_0\in \mathrm{int}(I)$,   the left derivative there exist (in $\mathbb{R}$)
$$
D^- \varphi(s_0):={\lim}_{s\to s_0 ^-} \; \frac{\varphi(s_0)-\varphi(s)}{s_0-s}
={\sup}_{s<s_0} \;\frac{\varphi(s_0)-\varphi(s)}{s_0-s}. \\
$$
Moreover,  if $s,t,r\in I$, $s<r$, and $s\leqslant t\leqslant r$ then $\varphi(t)-\varphi(s) \leqslant \left[\varphi(r)-\varphi(s)\right][(t-s)/(r-s)].$
\end{proposition}
%%%%%%%%%%%%%%%%%%%%%%%%%%%%%%%%%%%%%%%%%%%%%%%%%%%%%
\section{Robust Kantorovich's Theorem on Newton's Method} \label{SeC:RobKantTheo}
%%%%%%%%%%%%%%%%%%%%%%%%%%%%%%%%%%%%%%%%%%%%%%%%%%%%%
Our goal is to state and prove  a robust  affine invariant  version of Kantorovich's Theorem on Newton's Method for   finding  a zero of a vector field: \begin{equation} \label{eq:Inc}
X(p)=0,
\end{equation}
where $\cal M$ is a Riemannian manifold, $\Omega\subseteq {\cal M}$ an open set and $X:{\Omega}\to T{\cal M}$   a continuously differentiable  vector field.   The most important in our analysis is the robustness,  namely,   we give a  prescribed ball,  around the point  satisfying the Kantorovich's assumptions,  ensuring convergence of the method for any starting point in this ball. Moreover,  we  establish bounds for $Q$-quadratic convergence of  the method which depend  on the  majorant function.    Also, as in \cite{Alvarez2008} and  \cite{LiJinhua2008},  this analysis  allows us   establish existence and local uniqueness of the solution.  For state the theorem we need some definitions. We beginning with the following definition which was introduced in \cite{Alvarez2008}.
\begin{definition}\label{def:g2}
Let $R>0$, $n\in\mathbb{N}\backslash \{0\}$, $p_0\in {\cal M}$ and  ${\cal G}_n(p_0,R)$ be the class of all piecewise geodesic curves $\xi\ :[0,T] \to {\cal M}$  for some  $T>0$  which satisfy the following conditions:
\begin{enumerate}
\item $\xi (0)=p_0$ and the length of $\xi$ is no greater than $R$;
\item there exist $c_0, c_1,\ldots c_{n} \in [0,T]$ with $c_0=0 \leq c_1\leq \ldots \leq  c_n=T$ such that $\xi_{\mid_{[c_0,~c_1]}}$, \ldots $\xi_{\mid_{[c_{n-2}, ~c_{n-1}]}}$ are $n-1$ minimizing geodesics and $\xi_{\mid_{[c_{n-1},~c_n]}}$ is a geodesic.
\end{enumerate}
\end{definition}
\begin{remark}
Since  $\cal M$ is complete,  ${\cal G}_n(p_0,R)$ is nonempty.   Moreover, ${\cal G}_{n}(p_0,R) \subset {\cal G}_{n+1}(p_0,R)$ for all  $n\in\mathbb{N}\backslash \{0\}$.  Note that, in Definition~\ref{def:g2},    ${\cal G}_1(p_0,R)$ is the   class of  all  minimizing  geodesic curves $\xi :[0,T] \to {\cal M}$  with $\xi (0)=p_0$ and the length of $\xi$ is no greater than $R$.
\end{remark}
We also need the following definition  which was  equivalently  stated   in   (3.7) of    \cite{Alvarez2008}, for   ${\cal G}_2(p_0,R)$.
\begin{definition} \label{de:majcon}
Let   $\Omega\subseteq \cal M$ an open set and    $R>0$ a   scalar constanst.  A continuously differentiable $f:[0,\; R)\to \mathbb{R}$  is said to be a  majorant function   at a point  $p_0\in \Omega$ for   a   continuously   differentiable vector field $X:{\Omega}\to T{\cal M}$   with respect to ${\cal G}_n(p_0,R)$  if   $\nabla X(p_0)$  is nonsingular, $B(p_0,R) \subset \Omega$ and
\begin{equation} \label{eq:mc}
\left\|\nabla X(p_0)^{-1}\left[P_{\xi,b,0} \,\nabla X (\xi(b)) P_{\xi,a,b} -  \,P_{\xi,a,0}\nabla X (\xi(a))\right]\right\|
\leq f'\left(\ell[\xi, 0,b] \right)-f'\left(\ell[\xi,0,a]\right),
\end{equation}
 for all $\xi \in {\cal G}_n(p_0,R)$   with $a, b \in \mbox{dom}(\xi)$ and $0\leq a \leq b$. Moreover, $f$ satisfies the  following conditions:
\begin{itemize}
  \item[{\bf h1.}]  $f(0)>0$ and $f'(0)=-1$;
  \item[{\bf h2.}]  $f'$ is convex and strictly increasing;
  \item[{\bf h3.}] $f(t)=0$ for some $t \in (0,R)$.
\end{itemize}
\end{definition}
We also need of the following  condition on the majorant condition $f$  which will be considered to hold
only when explicitly stated
  \begin{itemize}
  \item[{\bf h4.}]  $f(t)<0$ for some $t\in (0,R)$.
  \end{itemize}
  \begin{remark}
  Since $f(0)>0$ and $f$ is continuous then   condition {\bf h4} implies condition {\bf h3}.
  \end{remark}
 The statement of our  main result  is:
\begin{theorem}\label{th:knt2}
Let $\cal M$ be a Riemannian manifold, $\Omega\subseteq {\cal M}$ an open set and   $\bar{\Omega}$ its  closure,  $X:\bar{\Omega} \to T{\cal M}$   a continuous vector field and continuously  differentiable on  $\Omega$,   $R>0$ a scalar constant and $ f:[0,R)\to \mathbb{R}$ a continuously differentiable function. Take $p_0 \in \Omega$.  Suppose that $\nabla X(p_0)$ is nonsingular  and     $f$ is a majorant function for $X$ at  $p_0$ with respect to ${\cal G}_3(p_0,R)$    satisfying  {\bf h4} and the inequality
 \begin{equation}  \label{KH.2}
   \left\|\nabla X(p_0)^{-1}X(p_0) \right\| \leq f(0).
  \end{equation}
Define   $\Gamma:=\sup\{ -f(t) ~:~ t\in[0,R)  \}$. Let $0\leq \rho<  \Gamma/2$  and   $g:[0,R-\rho)\to \mathbb{R}$,
\begin{equation} \label{eq:maj2}
  g(t):=\frac{1}{|f'(\rho)|}[f(t+\rho)+2\rho].
\end{equation}
Then $g$ has a smallest zero $t_{*,\rho}\in (0,R-\rho)$, the sequences  generated by Newton's Method for solving  the equation  $X(p)=0$ and  the equation  $g(t)=0$,  with starting  point $q_0$, for any $q_0\in  B[p_0,  \rho]$, and $t_0=0$, respectively,
  \begin{align}   \label{ns.KT2}
    q_{k+1}=\emph{exp}_{q_k}\left(- \nabla X(q_k)^{-1}X(q_k)\right),  \qquad \qquad    t_{k+1} ={t_k}- \frac{g(t_k)}{g'(t_k)}, \qquad  k=0,1,\ldots\,.
  \end{align}
   are well defined, $\{q_k\}$ is contained in $B(q_0,  t_{*,\rho})$, $\{t_k\}$ is strictly increasing, is contained in   $[0,t_{*,\rho})$ and  converges  to  $t_{*,\rho}$. Moreover,     $\{q_k\}$ and  $\{t_k\}$  satisfy  the inequalities
  \begin{equation}\label{eq:bdct2}
  d(q_{k}, q_{k+1})   \leq  t_{k+1}-t_{k}, \qquad k=0, 1, \ldots\, ,
  \end{equation}
  \begin{equation}\label{eq:qbdct2}
 d(q_{k}, q_{k+1})  \leq   \frac{t_{k+1}-t_{k}}{(t_{k}-t_{k-1})^2}  d(q_{k-1}, q_{k}) ^2\leq  \frac{D^-g'(t_{*,\rho})}{-2g'(t_{*,\rho})}  d(q_{k-1}, q_{k})^2,  \qquad k=1, 2, \ldots\,
  \end{equation}
and    $\{q_k\}$ converges  to $p_*\in B[q_0, t_{*,\rho}]$ such that  $X(p_*)=0$.   Furthermore,  $\{q_k\}$ and  $\{t_k\}$   satisfy the inequalities
\begin{equation}\label{eq:lcr2}
 d(q_{k}, p_*)   \leq  t_{*,\rho}-t_{k}, \qquad \qquad t_{*,\rho}-t_{k+1}\leq \frac{1}{2}( t_{*,\rho}-t_{k}), \qquad k=0,1, \ldots\,,
\end{equation}
 the convergence of $\{q_k\}$ and  $\{t_k\}$ to  $p_*$ and $t_{*,\rho}$, respectively,  are  $Q$-quadratic as follow
 \begin{equation}\label{eq:qcrt2}
\limsup_{k\to \infty}\frac{d(p_{k+1}, p_{*})} {d(q_{k}, p_*)^2}\leq   \frac{D^-g'(t_{*,\rho})}{-2g'(t_{*,\rho})},   \qquad t_{*,\rho}-t_{k+1} \leq \frac{D^-g'(t_{*,\rho})}{-2g'(t_{*,\rho})} ({t_{*,\rho}}-t_{k})^2,  \quad k=0,1, \ldots\,.
  \end{equation}
  and $p_*$ is the unique singularity  of $X$ in $B(p_0,{\bar \tau})$, where ${\bar \tau}\geq t_*$ is defined as
  \[ {\bar \tau}:=\sup\{ t\in [t_*, R)\;:\; f(t)\leq 0\}. \]
\end{theorem}
To prove the above theorem we need some previous results. First, in the next section, we prove a particular instance of this theorem, and then, in the Section~\ref{Sec:proof} we prove Theorem~\ref{th:knt2}.
%%%%%%%%%%%%%%%%%%%%%%%%%%%%%%%%%%%%%%%%%%%%%%%%%%%%%%%%%%%%%%%%%%%%%%%%%%
\section{Kantorovich's Theorem on Newton's Method} \label{Sec:KantTheo}
%\subsection{Preliminary results} \label{sec:PR}
%%%%%%%%%%%%%%%%%%%%%%%%%%%%%%%%%%%%%%%%%%%%%%%%%%%%%%%%%%%%%%%%%%%%%%%%%%
In this section we will prove an affine invariant  version of Kantorovich's Theorem on Newton's Method, it is a particular instance of Theorem~\ref{th:knt2}, namely,  the case $ \rho=0$.   We will use this theorem for proving Theorem~\ref{th:knt2}. The main results of this section are the bounds, depending  on the majorant function,  for the  $Q$-quadratic convergence of the Newton's Method, which gives an additional contribution for improving the results of   Alvarez et al. in \cite{Alvarez2008}, Ferreira and Svaiter in \cite{FerreiraSvaiter2002} and  Li and Wang in \cite{LiJinhua2008}.
\begin{theorem}\label{th:knt1}
Let $\cal M$ be a Riemannian manifold, $\Omega\subseteq {\cal M}$ an open set and   $\bar{\Omega}$ its  closure,  $X:\bar{\Omega} \to T{\cal M}$   a continuous vector field and continuously  differentiable on  $\Omega$,   $R>0$ a scalar constant and $ f:[0,R)\to \mathbb{R}$ a continuously differentiable function. Take $p_0 \in \Omega$.  Suppose that $\nabla X(p_0)$ is nonsingular  and     $f$ is a majorant function for $X$ at  $p_0$ with respect to ${\cal G}_2(p_0,R)$  satisfying the inequality
 \begin{equation}  \label{KH.1}
   \left\|\nabla X(p_0)^{-1}X(p_0) \right\| \leq f(0).
  \end{equation}
Then $f$ has a smallest zero $t_{*}\in (0,R)$, the sequences  generated by Newton's Method for solving  the equations  $X(p)=0$ and    $f(t)=0$,  with starting  point $p_0$  and $t_0=0$, respectively,
  \begin{align}   \label{ns.KT1}
    p_{k+1}=\emph{exp}_{p_k}\left(- \nabla X(p_k)^{-1}X(p_k)\right),  \quad    t_{k+1} ={t_k}- \frac{f(t_k)}{f'(t_k)}, \qquad  k=0,1,\ldots\,.
  \end{align}   are well defined, $\{p_k\}$ is contained in $B(p_0,  t_*)$, $\{t_k\}$ is strictly increasing, is contained in   $[0,t_*)$ and  converge  to  $t_*$ and   satisfy  the inequalities
\begin{equation}\label{eq:bd11}
  d(p_{k+1}, p_{k})   \leq  t_{k+1}-t_{k} , \qquad \qquad d(p_{k+1}, p_{k})  \leq   \frac{t_{k+1}-t_{k}}{(t_{k}-t_{k-1})^2}  d(p_{k}, p_{k-1})^2,
  \end{equation}
 for all \( k=0, 1, \ldots\, , \) and \( k=1,2, \ldots\, \), respectively. Moreover,    $\{p_k\}$ converge  to  $p_*\in B[p_0, t_*]$ such that  $X(p_*)=0,$
\begin{equation}\label{eq:lc11}
 d(p_*, p_{k})   \leq  t_*-t_{k}, \qquad \qquad t_*-t_{k+1}\leq \frac{1}{2}( t_*-t_{k}), \qquad k=0,1, \ldots\,
\end{equation}
and, therefore,    $\{t_{k}\}$  converges $Q$-linearly to $t_*$ and   $\{p_k\}$   converge $R$-linearly to $p_*$. If, additionally, $f$ satisfies {\bf h4} then the following inequalities hold:
\begin{equation}\label{eq:qcct1}
 d(p_{k+1}, p_{k}) \leq    \frac{D^-f'(t_*)}{-2f'(t_*)}   d(p_{k}, p_{k-1}) ^2,  \qquad t_{k+1}-t_{k} \leq \frac{D^-f'(t_*)}{-2f'(t_*)} ({t_k}-t_{k-1})^2, \qquad k=1,2, \ldots\,,
  \end{equation}
and, as  a consequence,    $\{p_k\}$ and  $\{t_k\}$  converge $Q$-quadratically
 to $p_*$ and $t_*$, respectively, as follow
 \begin{equation}\label{eq:qcs1}
\limsup_{k\to \infty}\frac{ d(p_*, p_{k+1})  } { d(p_*, p_{k})^2}\leq   \frac{D^-f'(t_*)}{-2f'(t_*)}, \qquad  \qquad t_{*}-t_{k+1} \leq \frac{D^-f'(t_*)}{-2f'(t_*)} ({t_*}-t_{k})^2,  \quad k=0,1, \ldots\,,
  \end{equation}
  and $p_*$ is the unique singularity  of $X$ in $B(p_0,{\bar \tau})$, where ${\bar \tau}\geq t_*$ is defined as
  \[ {\bar \tau}:=\sup\{ t\in [t_*, R)\;:\; f(t)\leq 0\}. \]
\end{theorem}
Henceforward  we assume that all assumptions in above theorem hold.
In this section, we will prove all the statements in Theorem~\ref{th:knt1} regarding to the majorant function and  the real sequence $\{t_k\}$ associated. The main relationships between the majorant function and the  vector field  will be also established.
%%%%%%%%%%%%%%%%%%%%%%%%%%%%%%
\subsection{The majorant function}
%%%%%%%%%%%%%%%%%%%%%%%%%%%%%%
In this subsection we will study the majorant function $f$ and prove all
results regarding only the real sequence $\{t_k\}$ defined by Newton's method applied to the majorant function $f$. Define
\begin{equation} \label{eq:def.bart}
  \bar{t}:=\sup \left\{t\in [0,R): f'(t)<0 \right\}\;.
\end{equation}
\begin{proposition} \label{pr:maj.f} The majorant function $f$  has a smallest root $t_*\in  (0,R)$, is strictly convex and
\begin{equation}  \label{eq:n.f}
   f(t)>0, \quad f'(t)<0, \qquad t<t-f(t)/f'(t)< t_*,  \qquad\qquad  \forall ~ t\in [0,t^*) .
\end{equation}
Moreover, $f'(t_*)\leqslant 0$ and
  \begin{equation}  \label{eq:pr.1.b}
     f'(t_*)<0\iff \exists\; t\in (t_*,R); \;f(t)< 0 .
  \end{equation}
  If, additionally, $f$ satisfies condition  {\bf h4} then the following statements  hold:
   \begin{itemize}
  \item[i)] $f'(t)<0$ for any $t\in [0,\bar{t}\,)$;
  \item[ii)] $0< t_* < \bar t\leq R$;
  \item[iii)] $0< \Gamma <\bar t$, where $\Gamma:=-\lim_{t\to \bar t_{-}} f(t)$.
    \item[iv)]  If  $0\leq \rho<\Gamma/2$  then $\rho<\bar t/2 <\bar t$ and  $f'(\rho)<0$.
\end{itemize}
\end{proposition}
\begin{proof}
See    Propositions 2.3 and 5.2 of \cite{FerreiraSvaiter2012} and Proposition 3 of \cite{FerreiraSvaiter2002}.
\end{proof}
In view of the second inequality in (\ref{eq:n.f}), Newton iteration is  well defined in $[0,t_*)$. Let us call it  $n_f:[0,t_*)\to \mathbb{R}$,
\begin{equation} \label{eq:n.f.2}
n_f(t):=t-f(t)/f'(t).
\end{equation}
\begin{proposition} \label{pr:2}
  Newton iteration  $n_f$ maps $[0,t^*)$ into $[0,t^*)$ and there hold:
  \begin{equation}
  t < n_f(t), \qquad t_*-n_f(t) \leq \frac{1}{2}(t_*-t),\qquad \forall ~ t \in [0,t_*).
  \end{equation}
  If $f$ also satisfies {\bf(h4)}, i.e., $f'(t_*) < 0$, then
  \begin{equation}
  t_*-n_f(t) \leq  \frac{D^-f'(t_*)}{-2f'(t_*)}(t_*-t)^2, \qquad \forall ~t \in [0,t^*).
  \end{equation}
\end{proposition}
\begin{proof}  See    Proposition 4 of \cite{FerreiraSvaiter2009}.
\end{proof}
The next two results follow from above proposition.
\begin{corollary} \label{cr:kanttauk}  Take any $\tau_0 \in [0,t_*)$ and define, inductively, $\tau_{k+1}=n_f(\tau_k)$, $k=0,1,... .$ The sequence $\{\tau_k\}$ is well defined, is strictly increasing, is contained in $[0,t_*)$ and converges $Q$-linearly to $t_*$  as follows
 \begin{equation*}
 t_*-\tau_{k+1}\leq \frac{1}{2}( t_*-\tau_{k}), \qquad k=0,1, \ldots\,
\end{equation*}
\end{corollary}
In particular,  the definition \eqref{ns.KT1} of $\{t_k\}$ in Theorem~\ref{th:knt1} is equivalent to the following one
\begin{equation}\label{eq:tknk}
  t_0=0,\quad t_{k+1}=n_f(t_k), \qquad k=0,1,\ldots\, .
\end{equation}
and there holds
\begin{corollary} \label{cr:kanttk}  The sequence $\{t_k\}$  is well defined, is strictly increasing, is contained in $[0,t_*)$ and converges $Q$-linearly to $t_*$  as follows
 \begin{equation*}
 t_*-t_{k+1}\leq \frac{1}{2}( t_*-t_{k}), \qquad k=0,1, \ldots\,
\end{equation*}
If $f$ also satisfies  {\bf h4}, then   the following  inequality  holds
\begin{equation}\label{eq:qcct}
 t_{k+1}-t_{k} \leq \frac{D^-f'(t_*)}{-2f'(t_*)} ({t_k}-t_{k-1})^2, \qquad k=1,2, \ldots\,,
  \end{equation}
and, as  a consequence,   $\{t_k\}$ converges  $Q$-quadratically to $t_*$   as follow
 \begin{equation}\label{eq:qcsk}
 t_{*}-t_{k+1} \leq \frac{D^-f'(t_*)}{-2f'(t_*)} ({t_*}-t_{k})^2,  \quad k=0,1, \ldots\,,
  \end{equation}
\end{corollary}
%%%%%%%%%%%%%%%%%%%%%%%%%%%%%%
\subsection{Relationship between the majorant function and the vector field} \label{sec:rbmno}
%%%%%%%%%%%%%%%%%%%%%%%%%%%%%%
In this subsection we will establish the main relationship between the majorant function and the vector field necessaries to prove Theorem~\ref{th:knt1}.
\begin{proposition}\label{ineq:inv}
Let $\xi \in {\cal G}_2(p_0,R)$. If $\ell[\xi,0,s]\leq t< \bar{t}$   then $\nabla X(\xi(s))$ is nonsingular and the following inequality holds
$$
\|\nabla X(\xi(s))^{-1}P_{\xi,0,s}\nabla X(p_0)\| \leq \frac{1}{\left |f'\left(\ell[\xi, 0,s] \right)\right|}  \leq \frac{1}{\left |f'\left(t\right)\right|}.
$$
\end{proposition}
\begin{proof}
Using Definition~\ref{de:majcon} and Lemma~\ref{lem:ban}, the proof  follows  the  same pattern of   Proposition 3.4 of \cite{FerreiraSvaiter2009}, see also  Lemma 4.2. of \cite{Alvarez2008}.
\end{proof}
Newton iteration at a point happens to be a zero of the linearization  at such a point.  Therefore, we study the linearization error of the vector field and the associated majorant function. The  formal definitions of these erros are:
\begin{definition} \label{de:errorf}
Let  $f:[0,\; R)\to \mathbb{R}$ be a  continuously differentiable function.  The linearization error of $f$ is defined by
\begin{equation} \label{eq:errorf}
e(a,b):=f(b)-\left[f(t)+f'(a)(b-a)\right], \qquad \forall \; a, b\in [0, R).
\end{equation}
\end{definition}
\begin{definition}  \label{de:errorX}
Let $\cal M$ be a Riemannian manifold, $\Omega\subseteq {\cal M}$ an open set, $X:{\Omega}\to T{\cal M}$   a continuously differentiable  vector field and $a, b\in [0, R)$. The  linearization error of $X$ on a geodesic  $\zeta:[a, b] \to \Omega$  is defined by
\begin{equation}\label{eq:errorX}
E\left(\zeta(a),\zeta(b)\right):=X(\zeta(b))-P_{\zeta,a,b}\left[X(\zeta(a))+(b-a)\nabla X(\zeta(a))\zeta'(a)\right].
\end{equation}
\end{definition}
In the next result we compare linearization error of the vector field with the linearization error of the majorant function associated.
\begin{lemma}\label{ineq:errors}
Let $\xi \in {\cal G}_2(p_0,R)$ be a curve passing through $p=\xi(a)$ and $q=\xi(b)$ such that   $\xi_{\mid_{[a,b]}}$ is a geodesic and  $0 \leq a \leq b$. Take  $0 \leq t < x< R$.  If $\ell[\xi,0,a]\leq t$ and $\ell[\xi, a,b] \leq x-t$, then
 $$
\left\|\nabla X(p_0)^{-1}P_{\xi,b,0} E(p,q) \right \| \leq e(t,x)\frac{\ell[\xi, a,b]^2 }{(x-t)^2}.
 $$
 As a consequence, the following inequality holds: $\left\|\nabla X(p_0)^{-1}P_{\xi,b,0} E(p,q) \right \| \leq e(t,x).$
\begin{proof}
Definition~\ref{de:errorX} with $\zeta=\xi_{\mid_{[a,b]}}$ and properties of parallel transport in \eqref{eq:ptp} imply
$$
E(p,q) = P_{\xi,a,b}\left[P_{\xi,b,a}X(q) - X(p) -(b-a)\nabla X(p)\xi'(a)\right].
$$
Hence, using Lemma~\ref{le:TFC} and that $\xi'(s)= P_{\xi,a,s} \xi'(a)$, the  last equality becomes
$$
E(p,q)  = P_{\xi,a,b} \int_{a}^{b}{\left[P_{\xi,s,a}\nabla X(\xi(s))P_{\xi,a,s} - \nabla X(p)\right]\xi'(a)\rm{d}s},
$$
which is equivalent to
$$
\nabla X(p_0)^{-1}P_{\xi,b,0}E(p,q)  = \int_{a}^{b}{ \nabla X(p_0)^{-1}\left[P_{\xi,s,0}\nabla X(\xi(s))P_{\xi,a,s} -  P_{\xi,a,0} \nabla X(p)\right]\xi'(a) \rm{d}s}.
$$
Since $\xi : [a,b] \to {\cal M}$ is a  geodesic  joining  $p$ and $q$ we have $\left\| \xi'(a)\right\|=\ell[\xi, a,b]/(b-a)$. Thus last equality implies
\begin{multline} \label{eq:errb1}
\left\|\nabla X(p_0)^{-1}P_{\xi,b,0}E(p,q) \right\|\leq  \\\int_{a}^{b} {\left\| \nabla X(p_0)^{-1}\left[P_{\xi,s,0}\nabla X(\xi(s))P_{\xi,a,s} -  P_{\xi,a,0} \nabla X(p)\right]\right\| \frac{\ell[\xi, a,b]}{b-a}\rm{d}s}.
\end{multline}
Because $a\leq s\leq b$, using   the assumptions $\ell[\xi,0,a]< t$ and $\ell[\xi, a,b] \leq x-t$ we have
$$
\ell[\xi, 0,s]\leq \ell[\xi, 0,a]+\ell[\xi, a,b]\leq x<R,
$$
and as   $\xi : [0, s] \to {\cal M}$   is  a   piecewise  geodesic curves joining the points $p_0$ to $\xi(s)$  through  $p$, i.~e.,  $\xi \in {\cal G}_2(p_0,R)$, we may use the  majorant condition in Definition~\ref{de:majcon} with  $b=s$ and $q=\xi(s)$ together with inequality in \eqref{eq:errb1} to conclude that
$$
\left\|\nabla X(p_0)^{-1}P_{\xi,b,0}E(p,q) \right\|\leq   \int_{a}^{b}\left[ f'\left(\ell[\xi, 0,s] \right)-f'\left(\ell[\xi,0,a]\right)\right]\frac{\ell[\xi, a,b]}{b-a}\rm{d}s.
$$
Using convexity of $f'$,   $\ell[\xi,0,a]\leq t$, $\ell[\xi, a,b]\leq x-t$, $x<R$ and  Proposition~\ref{pr:conv.aux1} we have
\begin{align*}
 f'\left(\ell[\xi, 0,s] \right)-f'\left(\ell[\xi,0,a]\right)
 &= f'\left(\ell[\xi, 0,a]+\ell[\xi,a,s] \right)-f'\left(\ell[\xi,0,a]\right)\\
 &\leq   f'\left(t+\ell[\xi, a,s] \right)-f'\left(t\right) \\
 &=  f'\left(t+\frac{s-a}{b-a}\ell[\xi, a,b]\right)-f'(t)\\
 &\leq  \left[f'\left(t+\frac{s-a}{b-a}(x-t)\right)-f'(t)\right]\frac{\ell[\xi, a,b]}{x-t}.
\end{align*}
Therefore, combining two last  inequality we obtain that
$$
\left\|\nabla X(p_0)^{-1}P_{\xi,b,0}E(p,q) \right\| \leq \int_{a}^{b}\left[f'\left(t+\frac{s-a}{b-a}(x-t)\right)-f'(t)\right]\frac{\ell[\xi, a,b]^2}{(x-t)(b-a)}\rm{d}s.
$$
After performing the integral and some algebraic manipulations the above inequality becomes
$$
\left\|\nabla X(p_0)^{-1}P_{\xi,b,0}E(p,q) \right\| \leq \left[f(x)-f(t)-f'(t)(x-t) \right]\frac{\ell[\xi, a,b]^2 }{(x-t)^2},
$$
which,  Definition~\ref{de:errorf},  implies the desired inequality.
\end{proof}
\end{lemma}
Proposition~\ref{ineq:inv} guarantees, in particular,  that   $\nabla X(p)$ is nonsingular  at $p\in B(p_0,t_*)$  and, consequently, the {\it Newton's iteration} is well defined in $ B(p_0,t_*)$.   Let us call it  $N_{X}: B(p_0,t_*)  \to \cal M$,
\begin{equation} \label{NF} \begin{array}{rcl}
    N_{X}(p):= \mbox{exp}_p(-\nabla X(p)^{-1}X(p)).
  \end{array}
\end{equation}
One can apply a \emph{single} Newton's iteration on any $p\in  B(p_0,t_*)$ to obtain the point $N_{X}(p)$
which may not is contained  to $ B(p_0,t_*)$, or even may not in the domain of $X$.
Hence,  this is enough to guarantee the well-definedness of only one iteration.   To ensure that Newtonian
iteration   may be repeated indefinitely, we need some additional definitions and results. First, we  define some subsets of $B(p_0, t_*)$ in which, as we shall
prove, Newton iteration \eqref{NF} is ``well behaved'':
\begin{align}\label{E:K}
K(t)&:=\left\{ p\in   \Omega ~:~ d(p_0, p)\leq t\, ,  ~ \; ~
   \left\|\nabla X(p)^{-1}X(p)\right\| \leqslant -\frac{f(t)}{f'(t)}\right\},\qquad
   t\in [0,t_*)\,.\\
  \label{eq:def.K}
 K&:=\bigcup_{t\in[0,t_*)} K(t),
\end{align}
 In  \eqref{E:K}, $0\leqslant t<t_*\leq \bar t$,  hence using  Proposition~\ref{pr:maj.f}  and  Proposition~\ref{ineq:inv}  we conclude that $f'(t)\neq 0$ and $\nabla X(p)$ is nonsingular in $B[p_0,t]\subset B[p_0,t_*)$, respectively. Therefore the above definitions are consistent. It is worth point out  that the above sets appeared for the first time in  \cite{FerreiraSvaiter2002}; see also  \cite{FerreiraSvaiter2009}.
\begin{lemma} \label{l:iset1}
For each $t\in [0, t_*)$ and each $p\in K(t)$  there hold:
\begin{itemize}
 \item[{\bf i)}]  $ \displaystyle \|\nabla X(p)^{-1}X(p)\| \leq  -\frac{f(t)}{f'(t)}$;
 \item[{\bf ii)}]  $d(p_0, p)+  \|\nabla X(p)^{-1}X(p)\| \leq n_{f}(t)<t_*$. As a consequence, $d(p_0,  N_{X}(p)) \leq n_{f}(t)<t_*.$
 \item[{\bf iii)}]  $ \displaystyle  \left\| \nabla X( N_{X}(p))^{-1}X( N_{X}(p))\right\| \leq  - \frac{f(n_f(t))}{f'(n_f(t))}\left[ \frac{\left\|\nabla X(p)^{-1}X(p)\right\|}{- f(t)/f'(t)}\right]^2$.
\end{itemize}
\end{lemma}
\begin{proof}
Let $t\in[0,t_*)$, $p\in K(t)$.  Using  definition of the set $K(t)$ in  \eqref{E:K}   the item {\bf i} follows.

Using  Proposition~\ref{pr:2} and definition of $K(t)$ in  \eqref{E:K}  to obtain that  $d(p_0, p)\leq t\, $  and $n_f(t)<t_*$, respectively. Hence,  the proof of the  first part of item {\bf ii} follows by combination of  two  last  inequalities  with      item {\bf i} and definition of $n_f$ in \eqref{eq:n.f.2}.   For proving the second part  of   item {\bf ii} use  triangular inequality to obtain  $d(p_0,  N_{X}(p))\leq d(p_0, p) + d(p,  N_{X}(p))$, definition in \eqref{NF} and  then first part.

We are going to prove item {\bf iii}. Let   $\xi:[0,2] \to {\cal M}$ a piecewise geodesic curve obtained by concatenation of a minimizing geodesic $\xi_{\mid_{[0,1]}}$ joining $p_0$ and $p$ and the geodesic curve   $\xi_{\mid_{[1,2]}}$ defined by
\begin{equation} \label{eq:newgeo}
 \xi(t)=exp_p\left((1-t)\nabla X(p)^{-1}X(p)\right).
\end{equation}
Note that $\xi\in {\cal G}_2(p_0,R)$. From  definition of the piecewise geodesic curve $\xi$ and definitions   in   \eqref{NF} and   \eqref{eq:newgeo}   we have
$$
\ell[\xi, 0,2]= d(p_0, p) +\left\|\nabla X(p)^{-1}X(p)\right\|.
$$
Since $\xi(2)=N_X(p)$, using last equality,   first inequality in  item {\bf ii}  and Proposition~\ref{ineq:inv},  by taking into account that   the derivative $f'$ is increasing and negative in $[0,  \bar t)$, we conclude that  $\nabla X(N_{X}(p) )$ is nonsingular and there holds
\begin{equation} \label{eq:blmainth}
\|\nabla X(N_{X}(p) )^{-1}P_{\xi,0,2}\nabla X(p_0)\| \leq \frac{1}{|f'(d(p_0, p)+\left\|\nabla X(p)^{-1}X(p)\right\|)|} \leq  \frac{1}{|f'(n_{f}(t))|}.
\end{equation}
On the other hand, as $\ell[\xi, 1, 2]=\left\|\nabla X(p)^{-1}X(p)\right\|$, combining item  {\bf i} with  definition of $n_f$ in \eqref{eq:n.f.2} we obtain $\ell[\xi, 1, 2]\leq n_f(t)-t$.  Since second part in item  {\bf ii}  imples  $d\left(p_0, N_{X}(p)\right) \leq n_{f}(t)<t_*$. Thus,  we may apply Lemma~\ref{ineq:errors} with $x=n_{f}(t)$ and  $q= N_{X}(p)$ to conclude that
\begin{equation} \label{eq:blmainth2}
\left\|\nabla X(p_0)^{-1}P_{\xi,2,0} E(p,N_{X}(p)) \right \| \leq e(t,n_{f}(t))\frac{\left\|\nabla X(p)^{-1}X(p)\right\|^2 }{(n_{f}(t)-t)^2}.
\end{equation}
We know that  $ N_{X}(p)$ belongs to the domain of  $X$. Hence,  Newton's iterations in  \eqref{NF},   linearization error in Definition~\ref{de:errorX}  with $\zeta=\xi_{\mid_{[1,2]}}$ and \eqref{eq:newgeo} yield\begin{equation*}
E(p,N_{X}(p)) = X(N_{X}(p))-P_{\xi,1,2}\left[X(p)+\nabla X(p)\left(-\nabla X(p)^{-1}X(p)\right)\right],
\end{equation*}
which is equivalent to
$
E(p,N_{X}(p)) = X(N_{X}(p)).
$
Thus, using this equality we obtain after  simples algebraic manipulation that
\begin{equation*}
\nabla X( N_{X}(p))^{-1}X( N_{X}(p))= \nabla X(N_{X}(p))^{-1}P_{\xi,0,2}\nabla X(p_0) \nabla X(p_0)^{-1}P_{\xi,2,0} E(p,N_{X}(p)).
\end{equation*}
Taking norm is last equality and using the inequalities \eqref{eq:blmainth} and \eqref{eq:blmainth2} we easily conclude  that
$$
\left\| \nabla X( N_{X}(p))^{-1}X( N_{X}(p))\right\| \leq \frac{e(t,n_{f}(t))}{|f'(n_{f}(t))|}\frac{\left\|\nabla X(p)^{-1}X(p)\right\|^2}{(n_{f}(t)-t)^2}.
$$
Finally, since  $n_f(t)$ belongs to the domain of  $f$, using the definitions of Newton iterations on \eqref{eq:n.f.2} and definition of the linearization error in \eqref{eq:errorf},  we obtain $ f(n_f(t))=e(t,n_f(t))$ which combined with $n_{f}(t)-t=f(t)/f'(t)$ and  last inequality implies the desired result. Therefore, the proof of the lemma is concluded.
\end{proof}

\begin{lemma} \label{NfNF}
For each $t\in [0, t_*)$ the following inclusions  hold: $K(t)\subset B(p_0,t_*)$ and
$$N_{X}\left( K(t) \right)\subset K\left( n_f(t) \right).$$
As a consequence,
$K\subset B(p_0,t_*)$ and $N_{X}(K)\subset K.$
\end{lemma}
\begin{proof}
The first inclusion follows trivially from the definition of $K(t)$ in \eqref{E:K}. Combining items {\bf i} and {\bf iii} of Lemma~\ref{l:iset1}  we have
 $$
\left\| \nabla X( N_{X}(p))^{-1}X( N_{X}(p))\right\| \leq  \frac{f(n_f(t))}{\left|f'(n_f(t))\right|}.
 $$
Therefore,  the second  inclusion of the lemma follows  from  combination of  last inequality in item {\bf ii} of Lemma~\ref{l:iset1},  last inequality and definition of $K(t)$. The first  inclusion on the second sentence follows trivially
from definitions \eqref{E:K} and \eqref{eq:def.K}.  To verify the last inclusion, take $p\in K$.  Then $p\in K(t)$ for some $t\in [0,t_*)$. Using the  first part of the lemma, we conclude that $N_{X}(p)\subseteq K(n_f(t))$. To end the proof, note that $n_f(t)\in [0,t_*)$ and use the definition of $K$ in \eqref{eq:def.K}.
\end{proof}
We  end this session  limiting the derivative of the vector field by the derivative of the majorant function.
\begin{proposition}\label{l:norm-X'(p)}
If  $d(p_0,p)\leq t < R$ then $\left\|\nabla X(p)\right\| \leq \|\nabla X(p_0)\|(2+f'(t))$.
\end{proposition}
\begin{proof}
Let $\xi:[0,1] \to {\cal M}$ is a minimizing geodesic joining $p_0$ to $p$. After some algebraic  manipulations we have
\begin{eqnarray*}
\left \| \nabla X(p_0)^{-1}P_{\xi,1,0}\nabla X(p)\right\| &=& \left \| \nabla X(p_0)^{-1}\left[P_{\xi,1,0} \nabla X(p)P_{\xi,0,1}-\nabla X(p_0) + \nabla X(p_0)\right]\right \|\\
& \leq & \left\| \nabla X(p_0)^{-1}\left[P_{\xi,1,0}\nabla X(p)P_{\xi,0,1}-\nabla X(p_0)\right]\right\|+ \left\| I_{p_0} \right \|.
\end{eqnarray*}
Since  $\xi$ is a minimizing geodesic joining $p_0$ to $p$ we have $\ell[\xi, 0,1]=d(p_0,p)$. Thus, using that  $f$ is  a majorant function at a point $p_0$ for the vector field $X$,  above inequality yelds
$$
\left \| \nabla X(p_0)^{-1}P_{\xi,1,0}\nabla X(p)\right\| \leq f'(d(p_0,p))-f'(0)+1 \leq  2 + f'(t),
$$
because  $d(p_0,p)\leq t $ and $f'$ is a increasing function.  Finally,  using last inequality and taking into account that
$$
\left\|\nabla X(p)\right\| \leq  \|\nabla X(p_0)\|\left \| \nabla X(p_0)^{-1}P_{\xi,1,0}\nabla X(p)\right\|,
$$
 the desired inequality follows.
\end{proof}
\subsection{Convergence}

In this  section  we establish  all the convergence results stated in Theorem~\ref{th:knt1} related to  $\{p_k\}$,   the sequence  generated by Newton's Method, namely, the  convergence of  $\{p_k\}$ to a zero of $X$, the bounds   in  \eqref{eq:bd11},  \eqref{eq:lc11},  \eqref{eq:qcct1}  and    \eqref{eq:qcs1}.   For establish these results we will combine conveniently the results of the previous section.   We begin with the following result:
\begin{proposition} \label{pr:qc}
Let $\{z_k\}$ be a sequence in $\cal M$ and  $C >0$. If  $\{z_k\}$ converges to $z_*$ and  satisfies
\begin{equation} \label{eq:qcfr}
d(z_k,z_{k+1})\leq C  d(z_{k-1},z_k)^2, \qquad k=1,2, \ldots.
\end{equation}
then $\{z_k\}$ converges $Q$-quadratically to $z_*$ as follows
$$
\limsup_{k\to \infty } \frac{d(z_{k+1},z_{*})}{d(z_{k},z_{*})^2}\leq C.
$$
\end{proposition}
\begin{proof}
The proof follows the same pattern as the proof of Proposition 1.2  of \cite{FG2013}.
\end{proof}
Using   equality in \eqref{ns.KT1} and \eqref{NF},  the sequence $\{p_k\}$    generated by Newton's Method satisfies
\begin{equation} \label{NFS}
 p_{k+1}= N_{X}(p_k),\qquad k=0,1,\ldots \,.
\end{equation}
This equivalent  definition of the Newton's sequence $\{p_k\}$ allow us to  use the results of the previous section to establishes its properties of  convergence.
\begin{corollary}\label{cor:knt1}
The sequence $\{p_k\}$  is well defined, is contained in $B(p_0,  t_*)$ and   satisfies  the inequalities in \eqref{eq:bd11}. Moreover,   $\{p_k\}$ converges   to  a point $p_*\in B[p_0, t_*]$  satisfying   $X(p_*)=0$ and its  convergence rate is  $R$-linear   as  in \eqref{eq:lc11}.  If, additionally, $f$ satisfies {\bf h4} then the inequality \eqref{eq:qcct1} holds and, consequently,    $\{p_k\}$ converges $Q$-quadratically  to $p_*$  as in  \eqref{eq:qcs1}.
\end{corollary}
\begin{proof}
We are going to prove that  the sequence $\{p_k\}$  is well defined.    First note that,   combining \eqref{E:K}, \eqref{KH.1} and {\bf h1} we have
\begin{equation} \label{eq:poinko}
p_{0} \in K(0)\subset K,
\end{equation}
where the second inclusion follows trivially from  \eqref{eq:def.K}.  Using the above inclusion,   the inclusion $N_{X}(K)\subset K$ in Lemma \ref{NfNF} and \eqref{NFS} we conclude that  $\{p_k\}$ is well defined and rests in $K $.  From the first inclusion on second part of the Lemma \ref{NfNF}
we have trivially that $\{p_k\}$ is contained in $B(p_0, t_*)$.

 Now we are going to prove the inequalities in \eqref{eq:bd11}. First we will prove, by induction that
\begin{equation} \label{eq:xktk}
        p_k\in K(t_k), \qquad k=0,1,\ldots \,.
\end{equation}
The above inclusion for  $k=0$  follows from   \eqref{eq:poinko}.
Assume now that $p_k\in K(t_k)$. Thus, using  Lemma~\ref{NfNF}, \eqref{NFS} and \eqref{eq:tknk},  we obtain that $p_{k+1}\in K(t_{k+1}),$ which completes the induction  proof of \eqref{eq:xktk}.  Using definition of $\{t_k\}$   in  \eqref{ns.KT1} , we have $- f(t_k)/f'(t_k)= t_{k+1}-t_{k}$. Hence combining   definition of $\{p_k\}$   in  \eqref{ns.KT1} with    \eqref{eq:xktk} and item~{\bf i} of Lemma~\ref{l:iset1}, we obtain
\begin{equation} \label{eq:eicm1}
d(p_k, p_{k+1})=\|\nabla X(p_k)^{-1}X(p_k)\| \leq t_{k+1}-t_k,  \qquad k=0,1,\ldots.
\end{equation}
which is   first inequality in    \eqref{eq:bd11}.   In order to prove  the second inequality in \eqref{eq:bd11}, first note that $  p_{k-1}\in K(t_{k-1})$, $p_{k}= N_{X}(p_{k-1})$  and $t_{k}=n_f(t_{k-1})$, for all $k=0,1,\ldots .$
Thus,  apply item~{\bf iii} of  Lemma~\ref{l:iset1}  with $p=p_{k-1}$   and $t=t_{k-1}$
to obtain
$$
d(p_k, p_{k+1})\leq -\frac{f(t_k)}{f'(t_k)}\left[ \frac{d(p_{k-1}, p_k)}{t_{k}-t_{k-1}}\right]^2,
$$
which using second inequality in \eqref{ns.KT1}  yields the desired  inequality.

To prove that   $\{p_k\}$ converges  to  $p_*\in B[p_0, t_*]$ with $X(p_*)=0$ and \eqref{eq:bd11} holds, first note that as $\{t_k\}$ converges to $t_*$, the first  inequality \eqref{eq:bd11} implies
\begin{equation} \label{eq:ch}
   \sum_{k=k_0}^\infty d(p_{k+1}, p_{k})   \leqslant
   \sum_{k=k_0}^\infty t_{k+1}-t_k =t_*-t_{k_0}<+\infty,
\end{equation}
for any $k_0\in\mathbb{N}$. Hence, $\{p_k\}$ is a Cauchy sequence in
$B(p_0, t_*)$ and ,thus, converges to some $p_*\in B[p_0,t_*]$.
Therefore,   first  inequality \eqref{eq:bd11} also implies  that $d(p_*, p_k)\leq t_*-t_k$ for any $k$. Hence, the inequality \eqref{eq:bd11} holds and,  as   $\{t_k\}$ converges $Q$-linearly to $t_*$,   $\{p_k\}$ converges $R$-linearly to $p_*$.   For  proving  that $X(p_*)=0$, note that first inequality in    \eqref{eq:bd11} implies that $d(p_0, p_k)\leq t_k-t_0=t_k$. Thus using Proposition~\ref{l:norm-X'(p)} we have
$$
 \left\|\nabla X(p_k)\right\| \leq \|\nabla X(p_0)\|(2+f'(t_k)) ,\qquad k=0,1,\ldots\,,
$$
which combining inclusion \eqref{eq:xktk} and second   inequality in \eqref{eq:eicm1} yields
$$
\|X(p_k)\|  \leq  \|\nabla X(p_k)\|\|\nabla X(p_k)^{-1}X(p_k)\|   \leq  \|\nabla X(p_0)\|(2+f'(t_k))(t_{k+1}-t_k),\quad k=0,1,\ldots\,.
$$
Since $X$ is  continuous on $\bar{\Omega}$, $\{p_k\}\subset B(p_0, t_*)  \subset \bar{\Omega}$, $\{p_k\}$  converges to $p_*\in \bar{\Omega}$, the result follows by taking limit as $k$ goes to infinite  in   above inequality.

Now,  we assume that {\bf h4} holds. Thus, combining  second inequality in \eqref{eq:bd11} with  \eqref{eq:qcct}, we obtain the  inequality in \eqref{eq:qcct1}. To establish the  inequality in  \eqref{eq:qcs1}, use inequality in \eqref{eq:qcct1} and  Proposition~\ref{pr:qc} with $z_k=p_k$ and $C=D^-f'(t_*)/(-2 f'(t_*))$. Therefore, the proof is concluded.

\end{proof}
%%%%%%%%%%%%%%%%%%%%%%%%%%%%%
\subsection{ Uniqueness}
%%%%%%%%%%%%%%%%%%%%%%%%%%%%%
In this section we prove the last statement in Theorem~\ref{th:knt1}, namely, the uniqueness of the singularity of the vector field in consideration.  The results  of this section  generalize   \cite[ Section~ 3.2 ]{FerreiraSvaiter2009} for a general majorant function, see also  \cite[ Section~ 4.2 ]{Alvarez2008}.
\begin{corollary}\label{cor:sing}
Take $0 \leq t < t_*$ and $q \in K(t)$. Define $$\tau_0=t, \qquad \tau_{k+1}=\tau_k-f(\tau_k)/f'(\tau_k),\qquad k=0,1,... .$$
The sequence $\{q_k\}$ generated by Newton's method with starting point $q_0=q$ is well defined and satisfies $q_k \in K(\tau_k)$,  for all k. Furthermore,   $\{\tau_k\}$ converges to $t_*$, $\{q_k\}$ converges to some $q_* \in B[p_0,t_*]$ a singular point of $X$ and $d(q_k,q_*) \leq t_*-\tau_k$,  for all $k$.
\end{corollary}
\begin{proof} The proof is a convenient combination of  Lemma \ref{NfNF},    Corollary~\ref{cr:kanttauk} and   Proposition~\ref{l:norm-X'(p)}, following the  same pattern of   Corollary 3.6 of  \cite{FerreiraSvaiter2009}.
\end{proof}
The next  two lemmas are most important results we need to  prove the uniqueness of solution.  The idea of its proofs are similar to the corresponding results of  \cite {FerreiraSvaiter2009}, see also \cite{Alvarez2008}. In  this more general approach, some technical details related to the parallel transport and the majorant function (possibly  non-quadractic) should be used.
\begin{lemma}\label{le:iterteta}Take $0 \leq t < t_*$ and $p \in K(t)$. Define for $\theta \in \mathbb{R}$
$$
\zeta(\theta)=exp_p(-\theta \nabla X(p)^{-1}X(p)), \qquad \quad  \tau(\theta) = t - \theta\dfrac{f(t)}{f'(t)}.
$$
Then for $\theta \in [0,1]$   we have $t \leq \tau(\theta) < t_*$ and $\zeta(\theta) \in K(\tau(\theta)).$
\begin{proof}
The proof follows  the  same pattern of   \cite[Lemma 3.7] {FerreiraSvaiter2009}, see also  \cite[Lemma 4.4]{Alvarez2008}.
\end{proof}
\end{lemma}
\begin{lemma}\label{le:iguald}
Take $0 \leq t < t_*$ and $p \in K(t)$. Suppose that $q_* \in B[p_0,t_*]$ is a singular point of $X$ and $t+d(p,q_*)=t_*.$ Then $d(p_0,p)=t.$ Furthermore,   $t <  n_f(t) < t_*$, $ N_{X}(p) \in K( n_f(t))$ and $n_f(t) + d(N_{X}(p), q_*)=t_*.$
\begin{proof}
The proof follows  the  same pattern of   \cite[Lemma 3.8] {FerreiraSvaiter2009}, see also  \cite[Lemma 4.5]{Alvarez2008}.
\end{proof}
\end{lemma}
The  proof of the next two results  can be obtained by a  simple adaptation of some arguments of   \cite[Corollary 3.9]{FerreiraSvaiter2009} and  \cite[Lemma 3.10]{FerreiraSvaiter2009}, see also  \cite[Lemma 4.5]{Alvarez2008} and  \cite[Section 4.2.2]{Alvarez2008},   we also omit their proofs.
\begin{corollary}\label{cor:iguald}
Suppose that $\tilde q_* \in B[p_0,t_*]$ is a singular point of $X$. If for some $\tilde{t},\tilde{q}$ $$0 \leq \tilde{t} < t_*, \,\,\,\tilde{q}\in K(\tilde{t}),$$
and $\tilde{t}+d(\tilde{q},\tilde q_*)=t_*,$ then $d(p_0, \tilde q_*)=t_*.$
\end{corollary}
\begin{lemma} \label{l:sbte}
The sequence $\{p_k\}$ has limit $p_*$ as the unique singular point of $X$ in $B[p_0,t_*]$.
\end{lemma}
\begin{lemma}\label{l:usp}
Let    $q \in B(p_0,R)$ and   $\xi: [0, 1]\to {\cal M}$  a minimizing geodesic  in  $ {\cal G}_1(p_0,R)$ joinning $p_0$ to $q$. Then the following inequality holds:
  $$
 -f(d(p_0,q)) \leq \|\nabla X(p_0)^{-1} P_{\xi,1,0}X(q)\|.
  $$
  As a consequence,   $p_*$ is the unique singularity  of $X$ in $B(p_0,{\bar \tau})$, where \( {\bar \tau}:=\sup\{ t\in [t_*, R)\;:\; f(t)\leq 0\} \).
\begin{proof}
 Applying second part of Lemma~\ref{ineq:errors} with $p=p_0$, $a=0$, $b=1$, $t=0$ and  $x=d(p_0, q)$ we have
$$
e(0,d(p_0, q)) \geq \left\|\nabla X(p_0)^{-1} P_{\xi,1,0}E(p_0,q)\right\|.
$$
From Definition~\ref{de:errorX}, last inequality becomes
$$
e(0,d(p_0, q)) \geq \left\|\nabla X(p_0)^{-1}P_{\xi,1,0}X(q)-\nabla X(p_0)^{-1}X(p_0)-\xi '(0)\right\|.
$$
Using  triangular inequality in the right hand side of last inequality, it is easy to see that
$$
e(0,d(p_0, q)) \geq \left \| \xi '(0)\right \| - \left \| \nabla X(p_0)^{-1}X(p_0)\right \| - \left \| \nabla X(p_0)^{-1}P_{\xi,1,0}X(q)\right \|.
$$
Combining  Definition~\ref{de:errorf} with assumption \eqref{KH.1} and taking into account that  $\left\| \xi '(0)\right\|=d(p_0, q)$  and $f'(0)=-1$, we obtain from last inequality that
\begin{equation*}
f(d(p_0,q))-\left[f(0)+f'(0)d(p_0, q)\right] \geq
d(p_0, q)-f(0)-\left \| \nabla X(p_0)^{-1} P_{\xi,1,0}X(q)\right \|,
\end{equation*}
with is equivalent to the inequality of the lemma. Hence  the first  of the lemma is proved.

For the second part, first  note that in the interval $(t_*, \bar \tau)$ the sign of $f$ is negative. Hence, first part of the lemma implies that there is no singularity of $X$ in  $B(p_0,{\bar \tau})\backslash B[p_0,t_*]$. Therefore, from Lemma~\ref{l:sbte}, the unique singularity of $X$ in $B(p_0,{\bar \tau})$ is  $p_*\in  B[p_0, t_*]$.
\end{proof}
\end{lemma}
\subsection{Proof of Theorem~\ref{th:knt1}} \label{Sec:proof}
The proof of Theorem~\ref{th:knt1} follow by direct  combination of  Corollary~\ref{cr:kanttk},   Corollary~\ref{cor:knt1} with Lemma~\ref{l:usp}.
%%%%%%%%%%%%%%%%%%%%%%%%%%%%%%%
\section{On the proof of the main theorem}  \label{pr:mr}
%%%%%%%%%%%%%%%%%%%%%%%%%%%%%%%
In this section Theorem~\ref{Sec:KantTheo} will be used  to  prove a robust semi-local affine invariant theorem for Newton's method for finding  a singularity of the vector field $X$, namely, Theorem~\ref{th:knt2}.  The following  result will be needed.
\begin{proposition}  \label{pr:ar2}
Let $R>0$ and $ f:[0,R)\to \mathbb{R}$ a continuously differentiable function.   Suppose that $p_0\in \Omega$,   $f$ is a majorant function for $X$ at  $p_0$  with respect to ${\cal G}_3(p_0,R)$  and satisfies {\bf h4}. If  $0\leq \rho<  \Gamma/2$, where   $\Gamma:=\sup\{ -f(t) ~:~ t\in[0,R)  \}$,  then for any $q_0\in  B[p_0,  \rho]$  the derivative    $\nabla X(q_0)$ is nonsingular. Moreover, the scalar  function    $g:[0,R-\rho)\to \mathbb{R}$,
\begin{equation*}
  g(t)=\frac{1}{|f'(\rho)|}[f(t+\rho)+2\rho],
\end{equation*}
is a majorant function for $X$ at $q_0$   with respect to ${\cal G}_2(q_0,R-\rho)$  and also satisfies condition {\bf h4}.
\end{proposition}
\begin{proof}
Since the domain of $f$ is  $[0, R)$  and $f'(\rho)< 0$ (see Proposition~\ref{pr:maj.f} item iv ), we conclude that $g$ is well defined.   First we will prove that function $g$ satisfies conditions {\bf h1, h2, h3} and {\bf h4}. Definition of $g$ and $f'(\rho)< 0$  trivially imply  $g'(0)=-1$.  Since $f$ is convex and $f'(0)=-1$ we have $f(t)+t\geq f(0)>0$,  for all  $0\leq t<R$, which,   by using    Proposition~\ref{pr:maj.f} item {\it iv} and that $0\leq \rho$, yields  $g(0)=  [f(\rho)+2\rho]/|f'(\rho)|>0$, hence $g$ satisfies {\bf h1}. Using that $f$ satisfies {\bf h2},  we easily conclude that $g$  also satisfies {\bf h2}.  Now,  as $\rho<  \Gamma/2$,  using Proposition~\ref{pr:maj.f}  item  {\it iii},   we have
\[
\lim_{t\to \bar t-\rho}g(t)=\frac{1}{|f'(\rho)|}(2\rho-\Gamma)<0\;,
\]
which implies that  $g$ satisfies {\bf h4} and, as $g$ is continuous and $g(0)>0$,  it also satisfies  {\bf h3}.

To complete the proof, it remains to prove that $g$ satisfies \eqref{eq:mc}. First of all,   for any   $q_0\in  B[p_0,  \rho]$, from   Proposition~\ref{pr:maj.f}   item {\it iv},   we have  $d(q_0, p_0)\leq \rho<\bar t$.  Let  $\eta:[0,1] \to {\cal M}$ be the minimizing geodesic joining $p_0$ to $q_0$. Since    $\eta \in {\cal G}_1(p_0,R)\subset {\cal G}_2(p_0,R)$  and  $d(p_0, q_0)=\ell[\eta,0,1]\leq \rho< \bar{t}$
 we can apply   Proposition~\ref{ineq:inv} to obtain  that  $\nabla X(q_0)$ is nonsingular and
\begin{equation} \label{eq:blnm}
\|\nabla X(q_0)^{-1}P_{\eta,0,1}\nabla X(p_0)\| \leq  \frac{1}{\left |f'\left(\rho\right)\right|}.
\end{equation}
Because $ B(p_0,  R) \subseteq \Omega$,  for any  $q_0\in  B[p_0,  \rho]$, we trivially  have   $B(q_0,R-\rho)\subset \Omega$.   Let
$\mu: [0, T] \to \cal{M}$ such that  $\mu \in {\cal G}_2(q_0,R-\rho)$ and  $c_0, c_1,c_{2} \in [0,T]$ with $c_0=0 \leq c_1 \leq  c_2=T$ such that $\mu_{\mid_{[c_0,~c_1]}}$  is  a minimizing geodesic and $\mu_{\mid_{[c_1,~c_2]}}$ is a geodesic. Take  $a, b \in [0, T]$ with  $0\leq a \leq b$. Thus
\begin{equation*}
\mu(a),~ \mu(b)\in B(q_0,R-\rho),\qquad \ell[\mu,0,a]+\ell[\mu,a,b]<R-\rho, \qquad d(q_0, \mu(a))=\ell[\mu,0,a].
\end{equation*}
Using definitions of the curves $\eta$ and $\mu$,  properties of the  parallel transport, property of the norm and simple manipulation, we conclude that
\begin{multline} \label{eq:finm}
\left\|\nabla X(q_0)^{-1}\left[P_{\mu,b,0} \,\nabla X (\mu(b)) P_{\mu,a,b} -  \,P_{\mu,a,0}\nabla X (\mu(a))\right]\right\|  \leq   \\
\left\|\nabla X(q_0)^{-1}P_{\eta,0,1}\nabla X(p_0)\right\|  \left\|\nabla X(p_0)^{-1}P_{\eta,1,0}\left[P_{\mu,b,0} \,\nabla X (\mu(b)) P_{\mu,a,b} -  \,P_{\mu,a,0}\nabla X (\mu(a))\right]\right\|.
 \end{multline}
     Now we are going to estimate the second  norm of the right hand side of above inequality. First,  we define  $\xi:[0,\hat{T}] \to \cal{M}$  a piecewise geodesic curve   in  $ {\cal G}_3(p_0,R)$ as concatenation between the curves $\eta$ and $\mu$, i.e., take $\hat{c}_0=0<\hat{c}_1<\hat{c}_2<\hat{c}_3=\hat{T}$ such that
\begin{equation} \label{eq:ccnm}
\xi_{\mid_{[\hat{c}_0,~\hat{c}_1]}}= \eta_{\mid_{[0,~1]}}, \qquad \xi_{\mid_{[\hat{c}_1,~\hat{c}_2]}}= \mu_{\mid_{[0,~c_1]}} , \qquad \xi_{\mid_{[\hat{c}_2,~\hat{c}_3]}}= \mu_{\mid_{[c_1,~c_2]}}.
\end{equation}
Definition of $\xi$ in \eqref{eq:ccnm} and  definition of curve $\mu$  imply that there exist        $\hat{a}, \hat{b} \in \mbox{dom}(\xi)$ with  $0\leq \hat{a} \leq \hat{b}$ such that $\xi(\hat{a})=\mu(a)$ and $\xi(\hat{b})=\mu(b)$.  Therefore,  properties of parallel transport  yield $P_{\eta,1,0} P_{\mu,b,0}= P_{\xi,\hat{b},0}$. Hence,
\begin{multline*}
 \left\|\nabla X(p_0)^{-1}P_{\eta,1,0}\left[P_{\mu,b,0} \,\nabla X (\mu(b)) P_{\mu,a,b} -  \,P_{\mu,a,0}\nabla X (\mu(a))\right]\right\|= \\ \left\|\nabla X(p_0)^{-1}\left[P_{\xi,\hat{b},0} \,\nabla X (\xi(\hat{b})) P_{\xi,\hat{a},\hat{b}} -  \,P_{\xi,\hat{a},0}\nabla X (\xi(\hat{a}))\right]\right\|.
 \end{multline*}
Since $\xi \in {\cal G}_3(p_0,R)$ and $f$ is a majorant function for $X$ at  $p_0$  with respect to ${\cal G}_3(p_0,R)$, applying  Definition~\ref{de:majcon} with $a=\hat{a}$ and $b=\hat{b}$,  last equality becomes
\begin{equation} \label{eq:mcnm}
 \left\|\nabla X(p_0)^{-1}P_{\eta,1,0}\left[P_{\mu,b,0} \,\nabla X (\mu(b)) P_{\mu,a,b} -  \,P_{\mu,a,0}\nabla X (\mu(a))\right]\right\|
\leq f'\left(\ell[\xi, 0,\hat{b}] \right)-f'\left(\ell[\xi,0,\hat{a}]\right).
\end{equation}
Combining last inequality with \eqref{eq:blnm}, \eqref{eq:finm} and \eqref{eq:mcnm} we obtain
\begin{multline} \label{eq:mcnmf}
\left\|\nabla X(q_0)^{-1}\left[P_{\mu,b,0} \,\nabla X (\mu(b)) P_{\mu,a,b} -  \,P_{\mu,a,0}\nabla X (\mu(a))\right]\right\|  \leq \\ \frac{1}{\left |f'\left(\rho\right)\right|} \left[ f'\left(\ell[\xi, 0,\hat{b}] \right)-f'\left(\ell[\xi,0,\hat{a}]\right) \right].
 \end{multline}
 Since $f'$ is convex, the function $s\mapsto f'(t+s)-f'(s)$ is increasing for $t\geq 0$. Hence taking into account that  definitions of $\xi$ in \eqref{eq:ccnm} and $\mu$   imply $\ell[\xi, 0,\hat{a}]= \ell[\xi, 0,\hat{c}_1]+ \ell[\xi, \hat{c}_1, \hat{a}] \leq \rho + \ell[\mu, 0, a]$ and   $\ell[\xi, 0, \hat{b}]= \ell[\xi, 0, \hat{a}]+ \ell[\xi, \hat{a}, \hat{b}] \leq \rho + \ell[\mu, 0, a] + \ell[\mu, a , b]$,   we conclude that
$$
f'\left(\ell[\xi, 0, \hat{b}] \right)-f'\left(\ell[\xi,0,\hat{a}] \right)\leq   f'(\rho + \ell[\mu, 0, a] + \ell[\mu, a , b])-f'(\rho + \ell[\mu, 0, a]).
$$
Since $\ell[\mu, 0, b]=\ell[\mu, 0, a] + \ell[\mu, a , b]$,  combining inequality in \eqref{eq:mcnmf} and  last  inequality with the definition of the function  $g$ we have
\[
\left\|\nabla X(q_0)^{-1}\left[P_{\mu,b,0} \,\nabla X (\mu(b)) P_{\mu,a,b} -  \,P_{\mu,a,0}\nabla X (\mu(a))\right]\right\|  \leq   g'\left(\ell[\mu, 0,b] \right)-g'\left(\ell[\mu,0,a]\right),
\]
implying that the function  $g$ satisfies \eqref{de:majcon}, which complete the proof of the proposition.
\end{proof}
\begin{proposition} \label{cor:majorX}
Let    $q \in B(p_0,R)$ and   $\xi: [0, 1]\to {\cal M}$  a minimizing geodesic   joinning $p_0$ to $q$. Then the following inequality holds:
\begin{equation}
\|\nabla X(p_0)^{-1} P_{\xi,1,0}X(q)\|\leq f(d(p_0,q))+2d(p_0,q).
\end{equation}
\end{proposition}
\begin{proof}
 Applying second part of Lemma~\ref{ineq:errors} with $p=p_0$, $a=0$, $b=1$, $t=0$ and  $x=d(p_0, q)$ we have
$$
e(0,d(p_0, q)) \geq \left\|\nabla X(p_0)^{-1} P_{\xi,1,0}E(p_0,q)\right\|.
$$
From Definition~\ref{de:errorX}   last inequality becomes
$$
e(0,d(p_0, q)) \geq \left\|\nabla X(p_0)^{-1}P_{\xi,1,0}X(q)-\nabla X(p_0)^{-1}X(p_0)-\xi '(0)\right\|.
$$
Using  triangular inequality in the right hand side of last inequality,  it is easy to see that
$$
e(0,d(p_0, q)) \geq  \left \| \nabla X(p_0)^{-1}P_{\xi,1,0}X(q)\right \|  - \left \| \nabla X(p_0)^{-1}X(p_0)\right \|-  \left \| \xi '(0)\right \| .
$$
Combining  Definition~\ref{de:errorf} with assumption \eqref{KH.1} and taking into account that  $\left\| \xi '(0)\right\|=d(p_0, q)$ and $f'(0)=-1$,  we obtain from last inequality that
\begin{equation*}
f(d(p_0,q))-\left[f(0)+f'(0)d(p_0, q)\right] \geq  \left \| \nabla X(p_0)^{-1} P_{\xi,1,0}X(q)\right \|-f(0)-d(p_0, q),
\end{equation*}
which is equivalent to the inequality of the lemma. Hence   the lemma is proved.
\end{proof}
%%%%%%%%%%%%%%%%%%%%%%%
\subsection{Proof of Theorem~\ref{th:knt2}}
%%%%%%%%%%%%%%%%%%%%%%%
Proposition~\ref{pr:ar2} claims that for any $q_0\in  B[p_0,  \rho]$  the derivative    $\nabla X(q_0)$ is nonsingular. Moreover, the scalar  function    $g:[0,R-\rho)\to \mathbb{R}$,
\begin{equation}\label{fmg}
  g(t)=\frac{1}{|f'(\rho)|}[f(t+\rho)+2\rho],
\end{equation}
is a majorant function for $X$ at $q_0$   with respect to ${\cal G}_2(q_0,R-\rho)$  and also satisfies condition {\bf h4}.  Let $\xi :[0, 1] \to {\cal M}$ a minimizing geodesic joining $p_0$ to $q_0$. Since item $iv$ of Proposition~\ref{pr:maj.f} implies $\ell[\xi, 0, 1]=d(p_0, q_0)\leq \rho<{\bar t}$, thus Proposition~\ref{ineq:inv}  give us
$$
\|\nabla X(q_0)^{-1}P_{\xi,0,1}\nabla X(p_0) \| \leq  \frac{1}{|f'(\rho)|}.
$$
Combining  property of norm with last inequality  and  Proposition~\ref{cor:majorX} with $q=q_0$,  we have
\begin{eqnarray*}
\|\nabla X(q_0)^{-1}X(q_0)\| & \leq & \|\nabla X(q_0)^{-1}P_{\xi,0,1}\nabla X(p_0) \| \|\nabla X(p_0)^{-1}P_{\xi,1,0}X(q_0)\| \\
& \leq & \frac{1}{|f'(\rho)|}[f(d(q_0,p_0))+2d(q_0,p_0)].
\end{eqnarray*}
As $f' \geq -1$, the function $t \mapsto f(t)+2t$ is (strictly) increasing.   Using  this fact,  above inequality, $d(p_0, q_0)\leq \rho$ and  (\ref{fmg}) we conclude that
$$
\|\nabla X(q_0)^{-1}X(q_0)\| \leq g(0).
$$
Therefore, last inequality allow us to apply Theorem~\ref{th:knt1} for $X$ and the majorant function $g$ at point $q_0$ for obtaining the desired result.

%%%%%%%%%%%%%%%%%%%%%%
\section{Special cases} \label{SEC:SpecCase}
%%%%%%%%%%%%%%%%%%%%%%
Kantorovich's theorem under a majorant condition  in Riemannian settings  was used in  \cite{Alvarez2008}, see also  \cite{LiJinhua2008} to prove  Kantorovich's theorem under Lipschitz  condition  in Riemannian  manifolds  \cite{FerreiraSvaiter2002},  Smale's theorem \cite{Smale1986}  and  Nesterov-Nemirovskii's theorem \cite{NesterovNemirovskii1994}.  Using the ideas of \cite{Alvarez2008}  we present,   as an application of  Theorem~\ref{th:knt2}, a   robust version of these theorems.
\subsection{Under Lipschitz's condition}
\begin{theorem}\label{th:kntlip}
Let $\cal M$ be a Riemannian manifold, $\Omega\subseteq {\cal M}$ an open set and   $\bar{\Omega}$ its  closure,  $X:\bar{\Omega} \to T{\cal M}$   a continuous vector field and continuously  differentiable on  $\Omega$. Take $p_0 \in \Omega$, $L>0$,  $\beta>0$ and $R=\sup\{ r>0~: ~B(p_0, r)\subset \Omega \}$. Suppose that $\nabla X(p_0)$ is nonsingular,  $B(p_0,  1/L)\subset \Omega$,
$$
\left\|\nabla X(p_0)^{-1}\left[P_{\xi,b,0} \,\nabla X (\xi(b)) P_{\xi,a,b} -  \,P_{\xi,a,0}\nabla X (\xi(a))\right]\right\|
\leq L\,\ell[\xi, a,b],
$$
for all $\xi$ in ${\cal G}_3(p_0, R)$ and $2\beta L < 1$. Moreover, assume that
$$
   \left\|\nabla X(p_0)^{-1}X(p_0) \right\| \leq \beta.
$$
Let $0\leq \rho<  (1-2\beta L)/(4L)$  and  $t_{*,\rho} =\left( 1-\rho L - \sqrt{1-2L(\beta+2\rho)}\right)/L$.  Then the sequence  generated by Newton's Method for solving  the equations  $X(p)=0$,  with starting  point $q_0$, for any $q_0\in  B[p_0,  \rho]$,
$$
    q_{k+1}=\emph{exp}_{q_k}\left(- \nabla X(q_k)^{-1}X(q_k)\right),  \qquad \qquad     k=0,1,\ldots\,.
$$
is well defined, $\{q_k\}$ is contained in $B(q_0,  t_{*,\rho})$ and   satisfy  the inequality
$$
 d(q_{k}, q_{k+1})  \leq  \frac{L}{2\sqrt{1-2L(\beta+2\rho)}}  d(q_{k-1}, q_{k})^2,  \qquad k=1, 2, \ldots\,
$$
Moreover,    $\{q_k\}$ converges  to $p_*\in B[q_0, t_{*,\rho}]$ such that  $X(p_*)=0$   and the convergence is   $Q$-quadratic as follows
$$
\limsup_{k\to \infty}\frac{d(q_{k+1}, p_{*})} {d(q_{k}, p_*)^2}\leq   \frac{L}{2\sqrt{1-2L(\beta+2\rho)}}.
$$
Furthermore,  if   $B(p_0,  \tau)\subset \Omega$  then $p_*$ is the unique singularity  of $X$ in $B(p_0,{\tau})$, where ${ \tau}:=\left( 1 + \sqrt{1-2\beta L}\right)/L.$
\end{theorem}
\begin{proof}
The proof follows from Theorem~\ref{th:knt2} with the quadratic polynomial $f(t)=\frac{L}{2}t^2-t+\beta$ as the majorant function to $X$ with respect to ${\cal G}_3(p_0,1/L)$ and $\Gamma = (1-2\beta L)/(4L)$.
\end{proof}
\subsection{Under Smale's condition}
\begin{theorem}
Let $\cal M$ be an analytic Riemannian manifold, $\Omega\subseteq {\cal M}$ an open set and    $X: \Omega \to T{\cal M}$   an analytic vector field. Let $p_0 \in {\cal M}$ be such that $\nabla X(p_0)$ is nonsingular and set  $\beta:=\left \|\nabla X(p_0)^{-1}X(p_0)\right \|$. Suppose
$$
\alpha:=\beta\gamma < 3-2\sqrt{2}, \qquad \quad \qquad \quad \gamma := \sup _{ n > 1 }\left\| \frac {1}{n !}\nabla X(p_0)^{-1}\nabla^n X(p_0)\right\|^{1/(n-1)}<\infty,
$$
$B(p_0,R)\subset \Omega$, where $R:=(1-1/\sqrt{2})/\gamma$. Let    $0 \leq \rho < [3-2\sqrt{2}-\alpha]/(2\gamma)$ and
$$
t_{*,\rho}:=\left(\alpha + 1 - 2\rho\gamma-\sqrt{(\alpha+1-2\rho\gamma)^2-8\alpha-8\rho\gamma(1-\alpha)}\right)/(4\gamma).
$$
Then the sequences generated by Newton's method for solving the equations $X(p)=0$ with starting at $q_0$, for any $q_0 \in B[p_0,\rho]$,
$$
q_{k+1}=exp_{q_k}(-\nabla X(q_k)^{-1}X(q_k)),\qquad  \qquad k=0,1,...
$$
are well defined, $\{q_k\}$ is contained in $B[q_0,t_{*,\rho}]$ and satisfy the inequality
$$
d(q_k,q_{k+1}) \leq \frac{\gamma}{(1-\gamma(t_{*,\rho}+\rho))[2(1-\gamma(t_{*,\rho}+\rho))^2-1]} d(q_{k-1},q_k)^2,  ,\qquad  \qquad k=1,2,...
$$
Moreover, $\{q_k\}$ converges to $p_* \in B[p_0,t_{*,0}]$ such that $X(p_*)=0$  and the convergence is  Q-quadratic as follows
$$\limsup_{k \to \infty}\frac{d(q_{k+1},p_*)}{d(q_k,p_*)^2} \leq \frac{\gamma}{(1-\gamma(t_{*,\rho}+\rho))[2(1-\gamma(t_{*,\rho}+\rho))^2-1]}.
$$
Furthermore,    $p_*$ is the unique singularity  of $X$ in $B\left(p_0, R\right)\subset \Omega$.
\end{theorem}
We need the following results to prove the above theorem.
\begin{lemma} \label{lemma:qc1}
Let $\cal M$ be an analytic Riemannian manifold, $\Omega\subseteq {\cal M}$ an open set and
  $X:{\Omega}\to T{\cal M}$ an analytic vector field.   Suppose that
$p_0\in \Omega$,  $\nabla X (p_0)$ is nonsingular and that $R \leq (1-1/\sqrt{2})\gamma^{-1}$. Then, for all $\zeta \in {\cal G}_{3}(p_0,R)$
there holds
$$
\|\nabla X (p_0)^{-1}P_{\zeta, s,0}\nabla^2 X(\zeta(s)))\| \leq  (2\gamma)/\left( 1- \gamma
\ell[\zeta,0,s]\right)^3.
$$
\end{lemma}
\begin{proof}
The proof follows the same pattern  of Lemma~5.3 of \textbf{\cite{Alvarez2008}}.
\end{proof}
\begin{lemma} \label{lc}
Let $\cal M$ be an analytic Riemannian manifolds, $\Omega\subseteq {\cal M}$ an open set and
  $X:{\Omega}\to T{\cal M}$ an analytic vector field.   Suppose that
$p_0\in \Omega$ and   $\nabla X (p_0)$ is nonsingular.  If there exists an \mbox{$f:[0,R)\to \mathbb {R}$} twice continuously differentiable such that
 \begin{equation} \label{eq:lc2}
\|\nabla X (p_0)^{-1}P_{\zeta, s,0}\nabla^2 X(\zeta(s)))\|\leqslant f''(\ell[\zeta,0,s]),
\end{equation}
for all $\zeta \in {\cal G}_3(p_0,R)$ and for all $s \in \mbox{dom}(\zeta)$, then $X$ and $f$
satisfy \eqref{eq:mc} with $n=3$.
\end{lemma}
\begin{proof}
Let $\zeta$ be a curve of ${\cal G}_3(p_0,R)$,  $a, b \in \mbox{dom}(\zeta)$ with $0\leq a \leq b$.  From Definition~\ref{def:g2} there exist  $c_0, c_1,c_2, c_{3} \in [0,T]$ with $c_0=0 \leq c_1\leq c_2  \leq  c_3=T$ such that $\xi_{\mid_{[c_0,~c_1]}}$ and  $\xi_{\mid_{[c_1,~c_2]}}$ are  minimizing geodesics and $\xi_{\mid_{[c_{2},~c_3]}}$ is a geodesic. We have six possibilities:
\begin{itemize}
\item $a, b \in [c_i,~c_{i+1}]$ for $i=0, 1, 2$;
\item $a \in [c_i,~c_{i+1}]$  and $ b \in [c_{i+1},~c_{i+2}]$ for $i=0, 1$;
\item $a \in [c_0,~c_{1}]$  and $ b \in [c_{2},~c_{3}]$.
\end{itemize}
We are going to analyze the possibility $a \in [c_0,~c_{1}]$  and $ b \in [c_{2},~c_{3}]$, the others  are similar.  Since  $a \in [c_0,~c_{1}]$  and $\xi_{\mid_{[c_0,~c_1]}}$ is geodesic, taking   $v\in T_{\zeta(a)}{\cal M}$ and $Y\in {\mathcal X}(\cal M)$ the  vector field on $\zeta$ such that  $ \nabla_{\zeta'(s)}Y = 0$ and $Y(\zeta(a))=v$, we may apply  Lemma~\ref{le:TFC2} to have
\begin{equation}\label{eq:ac1}
   P_{\zeta,c_1,a} \nabla X(\zeta(c_1))\,Y(\zeta(c_1))= \nabla X(\zeta(a))Y(\zeta(a)) + \int_a^{c_1}
   P_{\zeta,s,a} \nabla^{2} X (\zeta(s)) \left(Y(\zeta(s)), \zeta'(s)\right) \,ds.
\end{equation}
Using that  $Y(\zeta(a))=v$ and $Y(\zeta(c_1))=P_{\zeta,a,c_1}v$, we obtain, after some algebraic manipulation in last equality, that
\begin{multline*}
\nabla X(p_0)^{-1}\left[P_{\zeta,c_1,0}\nabla X(\zeta(c_1))P_{\zeta,a,c_1} - P_{\zeta,a,0}\nabla X(\zeta(a)) \right]v=  \\
\int_a^{c_1} \nabla X(p_0)^{-1} P_{\zeta,s,0} \nabla^{2} X (\zeta(s)) \left(Y(\zeta(s)), \zeta'(s)\right) \,ds.
\end{multline*}
Since $\|Y(\zeta(s))\|=\|v\|$ for all $s\in[a,c_1]$ and $v$ is a arbitrary,   we conclude from  Definition ~\ref{d:norm} that
\begin{multline*}
\left\|\nabla X(p_0)^{-1}\left[P_{\zeta,c_1,0}\nabla X(\zeta(c_1))P_{\zeta,a,c_1} - P_{\zeta,a,0}\nabla X(\zeta(a)) \right]\right\| \leq \\
  \int_a^{c_1}\|\nabla X (p_0)^{-1}P_{\zeta, s,0}\nabla^2 X(\zeta(s)))\|\|\zeta'(s)\|ds.
\end{multline*}
  Now, as  $\|\zeta'(s)\|=\ell[\zeta,a,c_1]/(c_1-a)$ and $\ell[\zeta,0,s]=\ell[\zeta,0,a]+((c_1-s)/ (c_1-a))\ell[\zeta,a,c_1]<R$ for all $s\in [a, c_1]$,  using \eqref{eq:lc2} we obtain, from the last inequality, that
\begin{multline*}
 \left\|\nabla X(p_0)^{-1}\left[P_{\zeta,c_1,0}\nabla X(\zeta(c_1))P_{\zeta,a,c_1} - P_{\zeta,a,0}\nabla X(\zeta(a)) \right]\right\| \leq \\
 \int_a^{c_1}f''\left(\ell[\zeta,0,a]+\frac{c_1-s}{c_1-a}\ell[\zeta,a,c_1]\right)\frac{\ell[\zeta,a,c_1]}{c_1-a}ds.
\end{multline*}
Evaluating the latter integral, it follows that
\begin{equation}\label{ineq:ac1}
\left\|\nabla X(p_0)^{-1}\left[P_{\xi,c_1,0} \,\nabla X (\xi(c_1)) P_{\xi,a,c_1} -  \,P_{\xi,a,0}\nabla X (\xi(a))\right]\right\|
\leq f'\left(\ell[\xi, 0,c_1] \right)-f'\left(\ell[\xi,0,a]\right).
\end{equation}
On the other hand, using that $\xi_{\mid_{[c_1,~c_2]}}$ is  geodesic, similar arguments used above   show that
\begin{equation}\label{ineq:c1c2}
\left\|\nabla X(p_0)^{-1}\left[P_{\xi,c_2,0} \,\nabla X (\xi(c_2)) P_{\xi,c_1,c_2} -  \,P_{\xi,c_1,0}\nabla X (\xi(c_1))\right]\right\|
\leq f'\left(\ell[\xi, 0,c_2] \right)-f'\left(\ell[\xi,0,c_1]\right).
\end{equation}
We may also use that  $b \in [c_2,~c_{3}]$  and $\xi_{\mid_{[c_2,~c_2]}}$  is geodesic to obtain the following inequality
\begin{equation}\label{ineq:c2b}
\left\|\nabla X(p_0)^{-1}\left[P_{\xi,b,0} \,\nabla X (\xi(b)) P_{\xi,c_2,b} -  \,P_{\xi,c_2,0}\nabla X (\xi(c_2))\right]\right\|
\leq f'\left(\ell[\xi, 0,b] \right)-f'\left(\ell[\xi,0,c_2]\right).
\end{equation}
Now,  taking into account that the parallel transport is an isometry, the triangular inequality yields
\begin{multline*}
\left\|\nabla X(p_0)^{-1}\left[P_{\xi,b,0} \,\nabla X (\xi(b)) P_{\xi,a,b} -  \,P_{\xi,a,0}\nabla X (\xi(a))\right]\right\| \leq \\
\left\|\nabla X(p_0)^{-1}\left[P_{\xi,b,0} \,\nabla X (\xi(b)) P_{\xi,c_2,b} -  \,P_{\xi,c_2,0}\nabla X (\xi(c_2))\right]\right\|+\\
\left\|\nabla X(p_0)^{-1}\left[P_{\xi,c_2,0} \,\nabla X (\xi(c_2)) P_{\xi,c_1,c_2} -  \,P_{\xi,c_1,0}\nabla X (\xi(c_1))\right]\right\|+\\
\left\|\nabla X(p_0)^{-1}\left[P_{\xi,c_1,0} \,\nabla X (\xi(c_1)) P_{\xi,a,c_1} -  \,P_{\xi,a,0}\nabla X (\xi(a))\right]\right\|.
\end{multline*}
Combining last inequality with \eqref{ineq:ac1},\eqref{ineq:c1c2} and \eqref{ineq:c2b}, it follows that
\begin{equation*}
\left\|\nabla X(p_0)^{-1}\left[P_{\xi,b,0} \,\nabla X (\xi(b)) P_{\xi,a,b} -  \,P_{\xi,a,0}\nabla X (\xi(a))\right]\right\|
\leq f'\left(\ell[\xi, 0,b] \right)-f'\left(\ell[\xi,0,a]\right),
\end{equation*}
which is the desired result.
\end{proof}
\begin{proof4}
Since $\alpha < 3-2\sqrt{2}$, combining Lemma~\ref{lemma:qc1} and Lemma~\ref{lc} we have that the analytic function $f:[0,R)\to \mathbb{R}$ defined by $f(t)=\beta-2t+t/(1-\gamma t)$ is a majorant function to $X$ with respect to ${\cal G}_3(p_0,R)$. Hence, the proof follows from Theorem~\ref{th:knt2} with  $\Gamma = (3-2\sqrt{2}-\alpha)/\gamma$.
\end{proof4}
\subsection{Under  Nesterov-Nemiroviskii's  condition}
\begin{theorem}
Let $C \subset \mathbb{R}^n$ be a open convex set and $F:C \to \mathbb{R}$ be a strictly convex function, three times continuously differentiable. Take $x_0 \in C$ with $F''(x_0)$ nonsingular. Define the norm
$$
\|u\|_{x_0}:=\sqrt{\langle u,u \rangle_{x_0}}, \qquad \forall ~ u \in \mathbb{R}^n,
$$
where $\langle u,v \rangle_{x_0}=a^{-1}\langle F''(x_0)u,v\rangle$ for all $u,v \in \mathbb{R}^n$ and some $a > 0$. Suppose that $F$ is $a$-self-concordant, i.e., satisfies
$$
|F'''(x)[h,h,h]| \leq 2a^{-1/2}(F''(x)[h,h])^{3/2}, \qquad \,\,\forall ~ x \in C, ~h \in \mathbb{R}^n,
$$
$W_1(x_0)=\{x \in \mathbb{R}^n:\|x-x_0\|_{x_0}<1\} \subset C$ and there exists $\beta \geq 0$ such that
$$
\|F''(x_0)^{-1}F'(x_0)\|_{x_0} \leq \beta < 3-2\sqrt{2}.
$$
Let $0 \leq \rho < (3-2\sqrt{2}-\beta)/2$ and $ t_{*,\rho}:=\left(\alpha + 1 - 2\rho-\sqrt{(\alpha+1-2\rho)^2-8\alpha-8\rho(1-\alpha)}\right)/4.$
 Then the sequences generated by Newton's method for solving the equations $F'(x)=0$  with starting at $y_0$, for any $y_0 \in W_\rho[x_0]=\{x \in \mathbb{R}^n:\|x-x_0\|_{x_0}\leq\rho\}  $,
$$y_{k+1}=y_k- F''(y_k)^{-1}F'(y_k),\qquad  \qquad k=0,1,...$$
is  well defined, $\{y_k\}$ is contained in $W_{t_{*,\rho}}[x_0]=\{x \in \mathbb{R}^n:\|x-x_0\|_{x_0}\leq t_{*,\rho}\} $  and satisfy the inequality
$$
\| y_{k+1}-y_k\|\leq\,\, \frac{1}{(1-(t_{*,\rho}+\rho))[2(1-(t_{*,\rho}+\rho))^2-1]}\| y_{k}-y_{k-1}\|^2, \qquad \quad k=1,2,... .
$$
Moreover, $\{y_k\}$ converges to $x_* \in W_{t_{*,0}}[x_0]$ such that $F'(x_*)=0$ and the convergence is  Q-quadratic as follows
$$
\limsup_{k \to \infty}\frac{\|x_*-y_{k+1}\|}{\|x_*-y_k\|^2} \leq \frac{1}{(1-(t_{*,\rho}+\rho))[2(1-(t_{*,\rho}+\rho))^2-1]}.
$$
\end{theorem}
\begin{proof}
Since $\alpha < 3-2\sqrt{2}$,   combining Lemma~$5.1$ of \cite{Alvarez2008} and Lemma~\ref{lc} we have that the  function $f:[0,R)\to \mathbb{R}$ defined by $f(t)=\beta-2t+t/(1- t)$ is a majorant function to $F'$ with respect to ${\cal G}_3(x_0,R)$. Hence, the proof follows from Theorem~\ref{th:knt2} with  $\Gamma = 3-2\sqrt{2}-b$.
\end{proof}

%%%%%%%%%%%%%%%%%%%%%%
\section{Final remark} \label{Sec:FinalRemarl}
%%%%%%%%%%%%%%%%%%%%%%
Let us present  some computational aspects of Newton's method in Riemmanin settings for solving the equation \eqref{eq:Inc}. Note that the first equality in \eqref{ns.KT2} is equivalent to
\begin{equation} \label{eq:nei}
q_{k+1}=\emph{exp}_{q_k}S_k,  \qquad  \nabla X(q_k)^{-1}S_k =-X(q_k), \qquad \qquad  k=0,1,....
\end{equation}
Since the solution of the linear systems in \eqref{eq:nei} for large systems is computationally expensive, namely, at each iteration the derivative at $q_k$ must be computed and stored. Besides, the solution of the linear system in \eqref{eq:nei} is required. To circumvent these drawbacks, we propose  the inexact Newton's method:  given an initial point $q_0$, the method generates a sequence $\{q_k\}$ as follows:
$$
q_{k+1}=\emph{exp}_{q_k}S_k,  \qquad  \nabla X(q_k)^{-1}S_k =-X(q_k)+r_k, \qquad  \|r_k\|\leq \theta_k \|X(q_k)\| \qquad  k=0,1,....
$$
for a suitable forcing sequence $\{\theta_k\}$, which is used to control the level of accuracy. Therefore, solutions of practical problems are obtained by computational implementations of the inexact Newton-like methods. The analysis of these methods under majorant condition will be done in the near  future.

%\bibliographystyle{abbrv}
%\bibliography{newton}

\end{document}